\documentclass[a4paper,12pt]{amsart}

\usepackage{latexsym}
\usepackage{amssymb}
\usepackage{graphicx}
\usepackage{amsthm}
\usepackage{epsfig}
\usepackage{float}
\RequirePackage{color}
\usepackage[normalem]{ulem}
\usepackage{mathabx}

\usepackage{amscd,enumerate}
\usepackage{amsfonts}

\input xy
\xyoption{all}

\usepackage{multicol}
\usepackage{hyperref}
\hypersetup{ colorlinks,
linkcolor=blue,
filecolor=green,
urlcolor=blue,
citecolor=blue }

\def\Z{\hbox{\hu Z}}

\newtheorem{lemma}{Lemma}[section]

\newtheorem{theorem}[lemma]{Theorem}
\newtheorem{proposition}[lemma]{Proposition}
\newtheorem{remark}[lemma]{Remark}

\newtheorem{example}[lemma]{Example}

\newtheorem{question}[lemma]{Question}

\newcommand{\Ker}{{\rm{Ker}}}

\def\Z{\mathbb Z}
\def\N{\mathbb N}

\def\soc{\mathop{\text{Soc}}}
\def\M{\mathop{\mathcal M}}

\definecolor{turquoise2}{rgb}{0,0.898039,0.933333}
\definecolor{}{rgb}{1,0,1}

\begin{document}


\author[M\"uge Kanuni]{M\"uge Kanuni}
\address{M\"uge Kanuni: Department of Mathematics. D\"uzce University, Konuralp 81620 D\"uzce, Turkey}
\email{mugekanuni@duzce.edu.tr}

\author[Dolores Mart\'{\i}n]{Dolores Mart\'{\i}n Barquero}
\address{Dolores Mart\'{\i}n Barquero: Departamento de Matem\'atica Aplicada, Escuela de Ingenier\'{\i}as Industriales, Universidad de M\'alaga. 29071 M\'alaga. Spain.}
\email{dmartin@uma.es}

\author[C\'andido Mart\'{\i}n]{C\'andido Mart\'{\i}n Gonz\'alez}
\address{C\'andido Mart\'{\i}n Gonz\'alez:  Departamento de \'Algebra Geometr\'{\i}a y Topolog\'{\i}a, Fa\-cultad de Ciencias, Universidad de M\'alaga, Campus de Teatinos s/n. 29071 M\'alaga. Spain.}
\email{candido\_m@uma.es}

\author[Mercedes Siles ]{Mercedes Siles Molina}
\address{Mercedes Siles Molina: Departamento de \'Algebra Geometr\'{\i}a y Topolog\'{\i}a, Fa\-cultad de Ciencias, Universidad de M\'alaga, Campus de Teatinos s/n. 29071 M\'alaga.   Spain.}
\email{msilesm@uma.es}

\thanks{
The first author is supported by {D\"uzce University Bilimsel Ara\c{s}t{\i}rma Projesi titled ``Leavitt, Cohn-Leavitt yol cebirlerinin ve C*-\c{c}izge cebirlerinin K-teorisi" with grant no: DUBAP-2016.05.04.462.}
The last three authors are supported by the Junta de Andaluc\'{\i}a and Fondos FEDER, jointly, through projects  FQM-336 and FQM-7156. 
They are also supported by the Spanish Ministerio de Econom\'ia y Competitividad and Fondos FEDER, jointly, through project  MTM2016-76327-C3-1-P.
\newline
This research took place while the first author was visiting the Universidad de M\'alaga. She thanks her coauthors for their hospitality.}

\title[Classification of Leavitt path algebras with two vertices]{Classification of Leavitt path algebras with two vertices}
\begin{abstract} 
We classify row-finite Leavitt path algebras associated to graphs with no more than two vertices. For the discussion we use the following invariants: decomposability, the $K_0$ group,  $\det(N'_E)$ (included in the Franks invariants), the type, as well as the
 socle, the ideal generated by the vertices in cycles with no exits and the ideal generated by vertices in extreme cycles. The starting point is a simple linear algebraic result that determines when a Leavitt path algebra is IBN.
 
 An interesting result that we have found is that the ideal generated by extreme cycles is invariant under any isomorphism (for Leavitt path algebras whose associated graph is finite).
 
 We also give a more specific proof of the fact that the shift move produces an isomorphism when applied to any row-finite graph, independently of the field we are considering.

\end{abstract}

\maketitle

\section{Introduction and preliminary results}
 
In 1960's Leavitt algebras arose from the work of Leavitt on his search for non-IBN algebras 
\cite{L}. The name Leavitt path algebras was associated to this structure, in particular, because 
the Leavitt path algebra on a graph with one vertex and $n$-loops, where $n >1$, is exactly the Leavitt algebra of type $(1,n)$. However, there are a lot of  Leavitt path algebras having IBN. For the definition of the type of a ring see, for example, \cite[Definition 1.1.1]{AAS}. 

The classification problem of Leavitt path algebras (up to isomorphisms) has been present in the literature since the pioneering works \cite{AALP} and \cite{Flow}. The study of the classification of Leavitt path algebras associated to small graphs was started in \cite{atlas}, where the authors considered graphs with at most $3$ vertices satisfying Condition (Sing), i.e, there is at most one edge between two vertices. This work can be also of interest, not only for people studying Leavitt path algebras, but also for a broader audience; concretely, for those working on graph C$^\ast$-algebras (as these are the analytic cousins of Leavitt path algebras). Moreover, one can view Leavitt path algebras as precisely those algebras constructed to produce specified $K$-theoretic data in a universal way, data arising naturally from directed graphs (sic \cite{AAS}), which could make these algebras and results a source of inspiration.

Throughout this paper we mean algebra isomorphism whenever we mention isomorphism. When referring to a ring isomorphism we will specify. In general, when there is a ring isomorphism between two algebras, these are not necessarily isomorphic as algebras. However, we will prove in Proposition \ref{RingVersusAlgebra} that when the center of the Leavitt path algebra is the ground field, then any ring isomorphism gives rise to an algebra isomorphism.

The goal of this article is the classification of Leavitt path algebras with at most two vertices and 
finitely many edges. This study will be initiated 
by fixing our attention in the IBN property, concretely, our starting point is \cite[Theorem 3.4]{KO}, which gives the necessary and sufficient condition that determines when a Leavitt path algebra has the IBN property in terms of a simple linear span of vertices. 
The outlay of the paper is as follows. Section 1 gives  the necessary preliminaries. Moreover, we give a detailed proof of the fact that shift move produces isomorphisms for row-finite graphs (see Theorem \ref{LuisFelipe}). We also prove in Proposition \ref{RingVersusAlgebra} that a ring isomorphism between two Leavitt path algebras whose center is the ground field produces an algebra isomorphism between them.
In Section 2 we compute the type of  Leavitt path algebras not having the IBN property via the criteria given in \cite{KO} and we give a first classification in Figure \ref{tableone}.
Section 3 contains the computation of the $K_0$-groups, which is stated in Figure \ref{tabletwo}. The main section of the paper, Section 4, follows the procedure of the decision tree given in Figure \ref{WeAreHappy} and discusses some algebraic invariants which are listed in Figures \ref{tablethree}, \ref{abanicoMuge} and \ref{tablethreebis}. The core resut, Theorem \ref{lagardeisillaBis}, classifies Leavitt path algebras not having the IBN-property. As a result of our research we prove in Theorem \ref{EnateMerlot-Merlot} that for a finite graph the ideal generated by the vertices in extreme cycles is invariant under ring isomorphisms.

Finally, in Section 5 we classify the Leavitt path algebras having the IBN property and conclude the sequel by addressing an open problem on the isomorphism of Leavitt path algebras over a particular pair of non-isomorphic graphs. In Theorem \ref{tintilladerotaBis} we classify Leavitt path algebras  having the IBN-property; the invariants we use are listed in Figures \ref{muge2} and \ref{muge3}.

\begin{center}\begin{figure}[H]
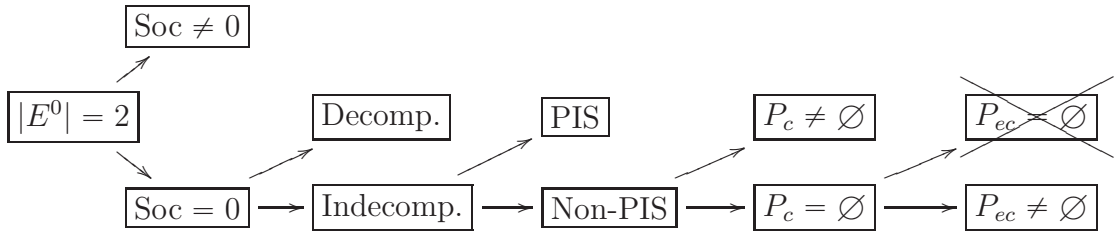
\label{WeAreHappy}
\xygraph{!{<0cm,0cm>;<1.5cm,0cm>:<0cm,1.2cm>::}
!{(0,0)}*+{{\fbox{$\vert E^0\vert =2$}}}="a"
!{(1,1)}*+{\fbox{$\hbox{Soc}\ne 0$}}="b"
!{(1,-1)}*+{\fbox{$\hbox{Soc}= 0$}}="c"
!{(2.7,0)}*+{\fbox{Decomp.}}="d"
!{(2.8,-1)}*+{\fbox{Indecomp.}}="e"
!{(4.4,0)}*+{\fbox{PIS}}="f"
!{(4.7,-1)}*+{\fbox{Non-PIS}}="g"
!{(6.5,0)}*+{\fbox{$P_c\ne\emptyset$}}="h"
!{(6.5,-1)}*+{\fbox{$P_c=\emptyset$}}="i"
!{(8.4,0)}*+{\fbox{$P_{ec}=\emptyset$}}="j"
!{(8.4,-1)}*+{\fbox{$P_{ec}\ne \emptyset$}}="k"
!{(7.7,0.5)}*+{}="l"
!{(9.2,-0.5)}*+{}="m"
!{(9.2,0.5)}*+{}="n"
!{(7.7,-0.5)}*+{}="p"
"a":"b"
"a":"c"
"c":"d"
"c":"e"
"e":"f"
"e":"g"
"g":"h"
"g":"i"
"i":"j"
"i":"k"
"l"-"m"
"n"-"p"
}
\medskip
\caption{Decision tree}
\end{figure}
\end{center}

Throughout the paper, $E = (E^0, E^1, s, r)$ will denote a directed graph with set of vertices $E^0$,  set of edges $E^1$, source function $s$, and range function $r$. In particular, the source vertex of an edge $e$ is denoted by $s(e)$, and the range vertex by $r(e)$. We call $E$ {\it finite}, if both $E^0$ and $E^1$ are finite sets and \emph{row-finite} if $s^{-1}(v)$ is a finite set for all $v\in E^0$.
A {\it sink} is a vertex $v$ for which  $s^{-1}(v) = \{e\in E^1 \mid s(e) = v\}$ is empty. 
For each $e\in E^{1}$, we call $e^{\ast}$ a \textit{ghost edge}. We let
$r(e^{\ast})$ denote $s(e)$, and we let $s(e^{\ast})$ denote $r(e)$. A \textit{path} $\mu$ of length $|\mu|=n>0$ is a finite sequence of edges $\mu
=e_{1}e_{2}\ldots e_{n}$ with $r(e_{i})=s(e_{i+1})$ for all
$i=1,\ldots,n-1$. In this case $\mu^{\ast}=e_{n}^{\ast}\ldots e_{2}^{\ast}e_{1}^{\ast}$ is the corresponding ghost path. A vertex
is considered a path of length $0$. The set of all vertices on the path $\mu$
is denoted by $\mu^{0}$. The set of all paths of a graph $E$ is denoted by ${\rm Path}(E)$. 
 
A path $\mu$ $=e_{1}\dots e_{n}$ in $E$ is \textit{closed} if $r(e_{n})=s(e_{1})$, 
in which case $\mu$ is said to be \textit{based at the vertex
}$s(e_{1})$. A closed path $\mu$ is called \textit{simple} provided that
it does not pass through its base more than once, i.e., $s(e_{i})\neq
s(e_{1})$ for all $i=2,\ldots,n$. The closed path $\mu$ is called a
\textit{cycle} if it does not pass through any of its vertices twice, that is,
if $s(e_{i})\neq s(e_{j})$ for every $i\neq j$.
 
An \textit{exit }for a path $\mu=e_{1}\dots e_{n}$ is an edge $e$ such that
$s(e)=s(e_{i})$ for some $i$ and $e\neq e_{i}$. We say the graph $E$ satisfies
\textit{Condition} (L) if every cycle in $E$ has an exit. We denote by $P_c^E$ ($P_c$ if there 
is no confusion about the graph) the set of vertices of a graph $E$ lying in
cycles without exits.

A cycle $c$ in a graph $E$ is called an \textit{extreme cycle} if $c$ has exits and for every path $\lambda$ starting at a vertex in $c$, there exists $\mu \in {\rm Path}(E)$, such that $0 \ne \lambda \mu$ and $r(\lambda \mu) \in c^0$. A {\it line point} is a vertex $v$ whose tree $T(v)$ does not contain any bifurcations or cycles. 
We will denote by $P^E_l$ the set of all line points, by $P^E_{ ec}$ the set of vertices which belong to extreme cycles, while $P^E_{ lec}:= P^E_l \sqcup P^E_{c} \sqcup P^E_{ ec}$. We will eliminate the superscript $E$ in these sets if there is no ambiguity about the graph.
We refer the reader to the book \cite{AAS} for other definitions and results on Leavitt path algebras.


\medskip
 
If there is a path from a vertex $u$ to a vertex $v$, we write $u\geq v$.
A subset $H$ of $E^{0}$ is called \textit{hereditary} if, whenever $v\in H$ and
$w\in E^{0}$ satisfy $v\geq w$, then $w\in H$. A  set $X$ is
\textit{saturated} if, for any  vertex $v$ which is not a sink, $r(s^{-1}(v))\subseteq X$
implies $v\in X$. The set of all hereditary saturated subsets of $E^{0}$ is denoted by
$\mathcal{H}_E$, which is also a partially ordered set by  inclusion. 

Let $K$ be a field, and let $E$ be a  row-finite graph. The {\it Leavitt path $K$-algebra} $L_K(E)$ {\em of $E$ with coefficients in $K$} is  the $K$-algebra generated by the set $\{v\mid v\in E^0\}$, together with  $\{e,e^*\mid e\in E^1\}$, which satisfy the following relations:

(V)  $vw = \delta_{v,w}v$ for all $v,w\in E^0$, \

(E1) $s(e)e=er(e)=e$ for all $e\in E^1$,

(E2)  $r(e)e^*=e^*s(e)=e^*$ for all $e\in E^1$, and

(CK1)  $e^*e'=\delta _{e,e'}r(e)$ for all $e,e'\in E^1$.

(CK2)  $v=\sum _{\{ e\in E^1\mid s(e)=v \}}ee^*$ for every $v\in E^0$ which is not a sink.

It was studied in \cite{KO} the necessary and sufficient conditions for a separated Cohn-Leavitt path algebra to have the Invariant Basis Number (IBN) property. In particular, when a Leavitt path algebra has IBN.  We refer the reader to \cite{AAS} for the definitions of separated graph, separated Cohn-Leavitt path algebra, etc.

The monoid of isomorphism classes of finitely-generated projective modules over a ring $A$ is denoted by $\mathcal{V}(A)$. 
Recall also that $\mathcal{U}(A)$ is the cyclic submonoid of $\mathcal{V}(A)$ generated by the isomorphism class of $A$.
The Grothendieck group of $\mathcal{V}(A)$ is the $K_0$-group of $A$ denoted $K_0(A)$, and by  \cite[Proposition 2.5]{KO}, there is
a monomorphism from the Grothendieck group of $\mathcal{U}(A)$  into $K_0(A)$.

By \cite[Theorem 3.5]{AMP} the abelian monoid $M_E$ associated with a  row-finite graph $E$ is  isomorphic to $\mathcal{V}(L_K(E))$. 
Concretely, when $E$ is finite, the isomorphism class of $L_K(E)$ is mapped to $[\sum_{v\in E^0}v]\in M_E$. Denote it by $[1]_E$.

Note that a Leavitt path algebra $L_K(E)$ which does not have IBN, necessarily has type $(1, m)$ for some natural number $m>1$. The reason is the following:  If $(n, m)$ were the type of $L_K(E)$ for  $1 < m \le n$, then $n[R]=m[R]$ and, by the separativity of the monoid $\mathcal{V}(L_K(E))$  (see \cite[Theorem 3.5 and Theorem 6.3]{AMP}), $(n-1)[R]=(m-1)[R]$, a contradiction to the type of the Leavitt path algebra.

For any  finite graph $E$, we denote by $A_E$ the incidence matrix of $E$. Formally, if $E^0 = \{v_i \mid 1 \leq i \leq n\}$, then  $A_E = (a_{i,j})$ is the $n \times n$ matrix for which $a_{i,j}$ is the number of edges $e$ having $s(e) = v_i$ and $r(e)=v_j$.   In particular, if $v_i \in E^0$ is a sink, then $a_{i,j} = 0$ for all $1\leq j \leq n$, i.e., the $i^{th}$ row of $A_E$ consists of all zeros.   
Following \cite{AB} we write $N_E$ and $1$ for the matrices in ${\mathbb Z}^{(E^0\times E^0\setminus
\text{Sink}(E))}$ obtained from $A^t_E$ and from the identity matrix after removing the columns corresponding
to sinks. Then there is a long exact sequence ($n\in {\mathbb Z}$)


$$\dots \to {\bf K}_n(K)^{(E^0\setminus {\rm Sink}(E))} \overset{1-N_E}{\longrightarrow} {\bf K}_n(K)^{(E^0)}
\longrightarrow {\bf K}_n(L_K(E)) \longrightarrow {\bf K}_{n-1}(K)^{(E^0\setminus {\rm Sink}(E))}.$$


In particular ${\bf K}_0(L_K(E)) \cong \text{coker}(1-N_E: {\mathbb Z}^{(E^0\setminus \text{Sink}(E))} \to
{\mathbb Z}^{(E^0)})$. The effective computation of the $K_0$ group of a given $L_K(E)$ is
explained in \cite[Section 3]{AALP}. 

Note that the $K_0(L_K(E))$ can be computed by obtaining the Smith normal form of the matrix $N'_E:=A_E-1'$, where $1'$ denotes the matrix built from the identity matrix changing the columns corresponding to sinks by columns of zeros. The element $[1]_E$, seen inside  $K_0(L_K(E))$, will be called the \emph{order unit}.

We will use intensively the Smith normal form of a matrix with entries in $\Z$.  Denote by $M_n(\Z)$ the ring of $n\times n$ matrices with integer coefficients. Following \cite{Smith}, for any matrix $A\in M_n(\Z)$  there are invertible matrices $P,Q$ in $M_n(\Z)$ such that $PAQ$ is a diagonal matrix
$PAQ=\text{diag}(d_1,\ldots,d_n)\in M_n(\Z)$, where $d_i\vert d_{i+1}$  and the diagonal entries are unique up to their signs. The diagonal matrix $PAQ$ is called the {\it Smith normal form of $A$}. 
\medskip

For the definition of the shift move we refer the reader to \cite[Definition 2.1]{AALP}. It was shown in \cite[Theorem 2.3]{AALP} that every shift of a graph $E$ produces an epimorphism between the corresponding Leavitt path algebras over a field $K$, which is an isomorphism provided the graph $E$ satisfies Condition (L) or the field $K$ is infinite. This result can be extended to arbitrary fields, and the condition can be eliminated, as the second, third and fourth author mentioned in \cite{atlas} (see page 583) and proved in a condensed way. Here we include a more detailed proof.

\begin{theorem}\label{LuisFelipe}
Let $K$ be an arbitrary field and let $E$ be a row-finite graph. Assume that $F$ is a graph obtained from $E$ by  shift moves. Then $L_K(E)$ and $L_K(F)$ are isomorphic.
\end{theorem}
\begin{proof}
Let $\varphi: L_K(E)\to L_K(F)$  be the $K$-algebra epimorphism defined in \cite[Theorem 2.3]{AALP}. Take $\overline K$, the algebraic closure of $K$, and consider the $K$-algebra homomorphism
$
\varphi\otimes 1_{\overline K}: L_K(E) \otimes  \overline K \to  L_K(F) \otimes \overline K
$, where $1_{\overline K}$ is the identity from $\overline K$ into $\overline K$. Since $\varphi$ and $1_{\overline K}$ are epimorphisms, then by \cite[Theorem 7.7]{SZE} the map $\varphi\otimes 1_{\overline K}$ is an epimorphism. The same result states that the kernel of $\varphi\otimes 1_{\overline K}$ is generated by $L_K(E) \otimes {\rm Ker}(1_{\overline K})\cup {\Ker (\varphi}) \otimes \overline K$; in fact, by ${\Ker (\varphi}) \otimes \overline K$, since ${\Ker}(1_{\overline K})=0$.

By \cite[Corollary 1.5.14]{AAS} we have that $L_{\overline K}(E)$ and $L_{\overline K}(F)$ are isomorphic to $L_K(E) \otimes  \overline K$ and to $L_K(F) \otimes \overline K$, respectively, via isomorphisms that we will denote by $\alpha$ and $\beta$, respectively. Therefore, there exists a unique $K$-algebra homomorphism $\overline \varphi$ that makes  the following diagram commute.
$$\xymatrix{ 
L_K(E)\otimes \overline K  \ar[dd]_(.6){\alpha}    \ar[rr]^{\varphi \otimes 1_{\overline K}} &                                                       & L_K(F)\otimes \overline K  \ar[dd]^(.6){\beta}           &                              \\ 
                                                                          & &                                                                &   \\
L_{\overline K}(E) \ar[rr]_{\overline\varphi}                                          &                                                        & L_{\overline K}(F)                                                        &
}
$$

Note that $\overline\varphi$ is just the $\overline K$-algebra homomorphism given in \cite[Theorem 2.3]{AALP}. Since $\overline K$ is an infinite field, this result states that $\overline \varphi$ is an isomorphism. By the commutativity of the diagram, the map $\varphi \otimes 1_{\overline K}$ is an isomorphism, therefore $\Ker (\varphi) \otimes \overline K= 0$. This implies $\Ker (\varphi)=0$, as required.
\end{proof}

We finish this section by including two results on isomorphisms which will be used in the sequel.

\begin{proposition}\label{RingVersusAlgebra}
 Let $E$ be a graph such that the center of $L_K(E)$ is isomorphic to $K$ (which implies that $E$ is a finite graph), and let $F$ be another graph. Then there is a ring isomorphism
$L_K(E)\to L_K(F)$ if and only if there is an algebra isomorphism
$L_K(E)\to L_K(F)$.
\end{proposition}
\begin{proof}
Assume that $f\colon L_K(E)\to L_K(F)$ is a ring isomorphism.  We can restrict the map $f$ to $f\vert_{Z(L_K(E))}\colon Z(L_K(E))\to Z(L_K(F))$, where $Z( \cdot)$ denotes the center of the algebra, to get an automorphism
$\sigma\colon K\to K$ such that $f(k 1)=\sigma(k) 1$ for any $k\in K$. We can say that $f$ is $\sigma$-linear in the sense that 
$f(k x)=\sigma(k)f(x)$ for any $k\in K$ and $x\in L_K(E)$. 

Now, by \cite[Corollary 1.5.12]{AAS}, we may fix a basis
$\{w_i\}_{i\in \Lambda}$ of $L_K(F)$ whose structure constants are $0,1,-1$. Assume $w_i w_j=\sum_l c_{ij}^l w_l$ where $c_{ij}^l\in\{0,\pm 1\}$. Define 
$\psi\colon L_K(F)\to L_K(F)$ by $\psi(\sum_i k_i w_i):=\sum_i\sigma^{-1}(k_i)w_i$, where  $k_i\in K$. This map is a $\sigma^{-1}$-linear bijective map and
$\psi(w_i w_j)=\psi(\sum_l c_{ij}^l w_l)=\sum_l c_{ij}^l w_l=w_iw_j=\psi(w_i)\psi(w_j)$. From this, we deduce that $\psi(xy)=
\psi(x)\psi(y)$ for any $x,y\in L_K(F)$. Thus the composition $\psi f$ is a $K$-linear isomorphism from $L_K(E)$ to $L_K(F)$.
\end{proof}

\begin{remark}
\rm
Although we have stated Proposition \ref{RingVersusAlgebra} for Leavitt path algebras, because we are in this setting, the result is more general: it is true for arbitrary $K$-algebras having center isomorphic to $K$ and a basis with structure constants in the prime field of $K$.
\end{remark}

\begin{proposition}\label{uzo}
Let $m$ and $n$ be natural numbers. Then, $M_\infty(L_K(1, m))$ is isomorphic to $M_\infty(L_K(1, n))$ if and only if $m=n$.
\end{proposition}
\begin{proof}
Assume that there is an isomorphism $\varphi: M_\infty(L_K(1, m)) \to M_\infty(L_K(1, n))$.
Let $e\in M_\infty(L_K(1, m))$ be the matrix having 1 in place 1,1  and zero everywhere else and let $e':=\varphi(e)$. Then
$L_K(1, m) \cong e M_\infty(L_K(1, m)) e \cong e'M_	\infty(L_K(1, n))e'$, which is Morita equivalent to $L_K(1, n)$. This implies that $L_K(1, m)$ is Morita equivalent to $L_K(1, n)$ and, consequently, their $K_0$ groups are isomorphic. Since the first one is isomorphic to $\Z_{m-1}$ and the second one is isomorphic to $\Z_{n-1}$, necessarily $m=n$.
\end{proof}

\section{Computation of the type of Leavitt path algebras not having IBN}
In this section we will determine all Leavitt path algebras not having the IBN property and compute their types in terms of the number of edges of the associated graphs. 

We start by quoting  \cite[Theorem 3.4]{KO}, which gives a necessary and sufficient condition for the algebra to have the IBN property in the more general setting of separated Cohn-Leavitt path algebras. 
\begin{theorem}\label{MainThmforseparatedLPA} 
For a given triple $(E, \Pi,\Lambda)$, with $E$ finite, let $L$ denote the separated Cohn-Leavitt path algebra
$CL_K(E, \Pi, \Lambda)$ over the triple. Then   
$L$ is IBN if and only if $ \sum_{v \in E^0}v$  is not in the $\mathbb{Q}$-span of the relations 
$\{ sX - \sum_{e \in X}r(e)\}_{ X \in \Lambda}$  in $\mathbb{Q} E^0$.
\end{theorem}

If the Leavitt path algebra has type $(1,m)$, for some natural $m>1$, then 
\begin{equation}\label{sardonia}
[1]= m[1],
\end{equation}
 where $m$ is the minimum natural number satisfying this property.
\medskip

Before we move on to the two-vertex graphs, for completeness of the argument, we state the easy case of one-vertex graphs. The Leavitt path algebras associated to one-vertex graphs are isomorphic either to the ground field $K$ or to the Laurent polynomial algebra $K[x,x^{-1}]$, which have the IBN property, or to the Leavitt algebras $L(1,n)$, with $n>1$, which do not have the IBN property and are of type $(1,n)$. 

\begin{figure}[h]
\begin{tabular}{|cc|c|c|c|}
\hline
 Graph & & $L_K(E)$ & IBN & Type \\
\hline
&&&&\\
$$ \xymatrix{
{\bullet}^u   }
$$ & & $K$ & YES & -\\
&&&&\\

 $$ \xymatrix{
{\bullet}^u\ar@(ur,dr)^{(1)}   }
$$& & $K[x,x^{-1}]$ & YES & - \\
&&&&\\
$$ \xymatrix{
{\bullet}^u\ar@(ur,dr)^{(n)}  }
$$& $ n\ge 2$ & $ L(1,n)$& NO &$(1,n)$\\
&&&&\\

\hline
\end{tabular}
\caption{All possible one-vertex Leavitt path algebras}\label{muge1}
\end{figure}

\bigskip

Let us consider a finite graph $E$ with two vertices and assume $l_1, l_2, t_1, t_2\in \N=\{0, 1, 2, \dots\}$ are the number of arrows appearing in the graph, that is,

$$ 
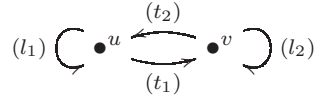
\xymatrix{
{\bullet}^u  \ar@(ul,dl)_{(l_1)} \ar@/^-.5pc/[r]_{(t_1)}   & {\bullet}^v
\ar@(ur,dr)^{(l_2)}    \ar@/_.5pc/[l]_{(t_2)}}
$$

Now, consider the set $\N\times \N$ and identify $u$ with $(1, 0)$ and $v$ with $(0, 1)$. 
According  to the number of  sinks in $E$, we have several different relations  in the monoid $M_E$.
If all the vertices are sinks, the graph consists of two isolated vertices and its Leavitt path algebra is $K\times K$ which has clearly the IBN property. 
So we  only consider the  two cases below.



\subsection{One sink case} \label{OneSinkCase} Without loss of generality, let $u$ be the sink  so that the graph looks like:

$$ 
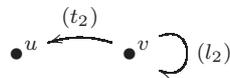
\xymatrix{
{\bullet}^u   & {\bullet}^v
\ar@(ur,dr)^{(l_2)}    \ar@/_.5pc/[l]_{(t_2)}}
$$
\medskip

Since there is only one  vertex which is not a sink, we have: 
\begin{equation}\label{onlyone}
v = t_2 u +l_2 v.
\end{equation}

\noindent Then $M_E$ is identified with 
$$\N\times\N\Big\slash\langle (0,1)=(t_2,l_2) \rangle$$ 

\noindent and we get  the equivalence relation generated by the pair 

\begin{equation}\label{onlyone}
 (t_2, l_2-1).
\end{equation}

A consequence of Theorem \ref{MainThmforseparatedLPA} is that the algebra $L_K(E)$ has not the IBN property and is of type $(1,m)$, $m>1$, if and only if $(m-1,m-1)$ is in the integer span of the pair in \eqref{onlyone}. In other words,
if and only if there is a nonzero natural number $k$ such that 
\begin{equation}\label{lacasera}
\begin{cases}m-1=k\ (l_2-1)\\ m-1=k\ t_2.\end{cases}
\end{equation}
We will split the discussion of the solution of this system into  two cases:
\begin{enumerate}
\item[Case 1.]  If $t_2=0$ or $l_2=1$ then the system is inconsistent  (it has no solution) and $L_K(E)$ has IBN. More precisely: when $t_2=0$, then $L_K(E)$ is isomorphic to $K \times K[x, x^{-1}]$ when $l_2=1$ or to $K \times L(1, l_2)$ when $l_2\neq 1$ and in every case it is an IBN algebra.
\item[Case 2.] If $t_2\ne 0$ and $l_2\ne 1$ then:
\begin{enumerate}
\item If $t_2\ne l_2-1$, then the system is again inconsistent and $L_K(E)$ is IBN.
\item If $t_2= l_2-1$, then  $m-1=k\ t_2$ and the minimum solution is $m=1+t_2=l_2$, so $L_K(E)$ has not IBN and type $(1,l_2)$.
\end{enumerate}
\end{enumerate}

Summarizing the results of the one sink case, we get the following lemma. 

\begin{lemma} \label{oporto} Let $K$ be a field and $E$ be a graph with two vertices having exactly one sink. 
Then $L_K(E)$ has not IBN  if and only if $E$ is of the form  
$$ 
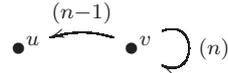
\xymatrix{
{\bullet}^u   & {\bullet}^v
\ar@(ur,dr)^{(n)}    \ar@/_.5pc/[l]_{(n-1)}}
$$  
where $n\ge 2$. Furthermore, $L_K(E)$ has not IBN and has type $(1,n)$.
\end{lemma}

\subsection{No sink case}\label{Eristoff}

If both $u$ and $v$ are  not sinks, then we have the following relations
\begin{equation}\label{muga}
\begin{aligned}
u &= l_1 u + t_1 v,\\
v&= t_2 u +l_2 v.
\end{aligned}
\end{equation}

Then $M_E$ can be identified with 
$$\N\times\N\Big\slash \langle (1, 0) =(l_1, t_1), (0, 1)= (t_2, l_2) \rangle$$


and the equivalence relation is the one generated by the pairs 

\begin{equation}\label{geva}
(l_1-1, t_1), \quad (t_2, l_2 -1).
\end{equation}

Using Theorem \ref{MainThmforseparatedLPA} we can affirm that if there exists $m\in \N$, $m>1$ such that \eqref{sardonia} is satisfied, then $(m, m)-(1, 1)$ is in the $\Z$-span of the relations given in \eqref{geva}, that is, there exist $k_1, k_2 \in \Z$ such that

\begin{equation}\label{matarromera}
\begin{aligned}
\begin{cases}
m-1 &= k_1(l_1-1) + k_2 t_2\\
m-1&= k_1t_1 + k_2(l_2-1).
\end{cases}
\end{aligned}
\end{equation}

Our aim is to find the minimum value of $m\in \N$ satisfying the system above, if it exists.
\medskip

We may assume that $(l_i, t_i)\ne (0,0)$ for  any $i=1,2$ (otherwise  the graph has a sink and this case has been considered already).
There are also some particular cases to consider:
\begin{enumerate}
\item[{\bf Case 1.}] $(l_i,t_i)=(1,0)$ for any $i=1,2$. The associated Leavitt path algebra is isomorphic to $K[x,x^{-1}]\times K[x,x^{-1}]$ which has IBN.
\item[{\bf Case 2.}]  Without loss of generality we may assume $(l_1,t_1)=(1,0)$ but $(l_2,t_2)\ne (1,0)$.  The graph is
$$ 
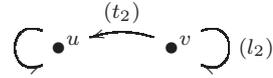
\xymatrix{
{\bullet}^u  \ar@(ul,dl)   & {\bullet}^v
\ar@(ur,dr)^{(l_2)}    \ar@/_.5pc/[l]_{(t_2)}}
$$
\smallskip

and the system transforms to
$$\begin{aligned}
\begin{cases}
m-1 &=  k_2 t_2\\
m-1&=  k_2(l_2-1),
\end{cases}
\end{aligned}
$$
\noindent which is the same  system as \eqref{lacasera}. Consequently, the algebra does not have IBN and has type $(1,l_2)$ if and only if $t_2=l_2-1$.
\end{enumerate}
\medskip

Here, we get another class of  Leavitt path algebras not having  IBN and we note the result as the following lemma. 
\begin{lemma} \label{ribeiro} Let $K$ be a field and $E$ be a graph of the form
 $$ 
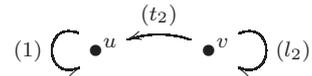
\xymatrix{
{\bullet}^u  \ar@(ul,dl)_{(1)}   & {\bullet}^v
\ar@(ur,dr)^{(l_2)}    \ar@/_.5pc/[l]_{(t_2)}}
$$
where $l_2\ge 2$.
Then $L_K(E)$ does not have IBN if and only if $t_2=l_2-1$. In this case, $L_K(E)$ has type  $(1,l_2)$.  
\end{lemma}

\begin{enumerate}
\item[{\bf Case 3.}] $(l_i,t_i)\ne (1,0)$ for any $i=1,2$. 
\end{enumerate}

\noindent
{\bf Case 3a}. If $l_i-1=t_i$ for some $i$. Without loss of generality, assume $l_1-1=t_1$. From \eqref{matarromera} we get $0=k_2(t_2-l_2+1)$. 

\noindent
{\bf Case 3a (i).} If
$t_2=l_2-1$, then we have $m-1=k_1t_1+k_2t_2$,  therefore $m-1=k\ \gcd(t_1,t_2)$ and the minimum solution  (in $\N$) is $m=1+\gcd(t_1,t_2)$.  We get a Leavitt path algebra not having IBN and of type $(1,1+\gcd(t_1,t_2))$. 
\medskip

We state this result in the following lemma.

\begin{lemma} \label{madeira} Let $K$ be a field and $E$ be a graph of the form
 $$ 
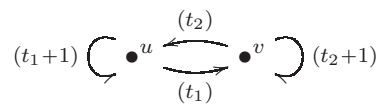
\xymatrix{
{\bullet}^u  \ar@(ul,dl)_{(t_1+1)} \ar@/^-.5pc/[r]_{(t_1)}  & {\bullet}^v
\ar@(ur,dr)^{(t_2+1)}    \ar@/_.5pc/[l]_{(t_2)}}
$$
where  $t_1,t_2\ge 1$. Then $L_K(E)$ does not have IBN and has type $(1,1+gcd(t_1,t_2))$. 
\end{lemma}

\noindent
{\bf Case 3a (ii).}  If $t_2\ne l_2-1$ we
have $k_2=0$ and the minimum solution for $m$ is $m=1+t_1=l_1$, which gives  a Leavitt path algebra not having IBN and of type $(1,l_1)$.
\medskip

This case is summarized below.

\begin{lemma} \label{albarino} Let $K$ be a field and $E$ be a graph of the form
 $$ 
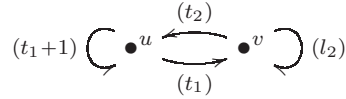
\xymatrix{
{\bullet}^u  \ar@(ul,dl)_{(t_1+1)} \ar@/^-.5pc/[r]_{(t_1)}  & {\bullet}^v
\ar@(ur,dr)^{(l_2)}    \ar@/_.5pc/[l]_{(t_2)}}
$$ 
where $(l_2,t_2)\ne (0,0)$,  $l_2 -t_2 \neq 1$ and $t_1 \geq 1$. 
Then $L_K(E)$ does not have IBN and has type $(1,1+t_1)$.  
\end{lemma}

\noindent
{\bf Case 3b.} We analyze now the case $l_i-1\ne t_i$ for any $i$.

In what follows we will recall how to solve the following system of equations on $\Z$. Let $a, b, a', b', c, c' \in \Z$ be such that at least one of the following elements: $a, a', b, b'$ is non zero, and consider:
\begin{equation}\label{faustino}
\begin{aligned}
\begin{cases}
c&=\  a k_1 + b k_2\\
c'&= \ a' k_1 + b' k_2.
\end{cases}
\end{aligned}
\end{equation}

Without loss in generality we may assume $a$ or $a'$ is different from zero. 

Since $a$ or $a'$ is nonzero, we may define $d:= \gcd(a, a')$, which is nonzero. We know that there exist $s, t\in \Z$ such that $d=as +a't$.

Now \eqref{faustino} can be rewritten as the matrix equation

\begin{equation}\label{santimia}
\begin{pmatrix}
a & b \\
a' & b'
\end{pmatrix}
\begin{pmatrix}
k_1 \\ k_2
\end{pmatrix} = \begin{pmatrix}
c \\ c'
\end{pmatrix}.
\end{equation}

A simple computation shows that the matrix 
$A=\begin{pmatrix}
s & t \\
-{\frac{a'}{d}} & {\frac{a}{d}}
\end{pmatrix}$  
is invertible with inverse 
$\begin{pmatrix}
\frac{a}{d} & -t \\
\frac{a'}{d} & s
\end{pmatrix}$. Multiplying (\ref{santimia}) by $A$ on the left hand side we get 

\begin{equation}\label{otro}
\begin{pmatrix}
d & \ast \\
0 & {\frac{\Delta}{d}}
\end{pmatrix}
\begin{pmatrix}
k_1\\ k_2
\end{pmatrix} = A\begin{pmatrix} 
c \\ c'
\end{pmatrix} = 
\begin{pmatrix}
sc+tc' \\ 
{-\frac{a'c}{d}}+ {\frac{ac'}{d}}
\end{pmatrix},
\end{equation}
where $\Delta=ab'-a'b$.

Consequently
 ${\frac{\Delta}{d}} k_2 = {\frac{ac'-a'c}{d}}$, implying ${\Delta} k_2 = {ac'-a'c}$. By performing the following substitutions 
$a=l_1-1$, $a'=t_1$, $b=t_2$, $b'=l_2-1$, $c=c'=m-1>0$, the last equation becomes 

\begin{equation}\label{Donsimon}
\Delta k_2=(l_1-1-t_1)(m-1),
\end{equation}
\noindent where $\Delta=(l_1-1)(l_2-1)-t_1t_2$. By swapping the roles of the vertices and  following a similar argument, we get
\begin{equation}\label{Donsimondos}
\Delta k_1=(l_2-1-t_2)(m-1).
\end{equation}
\smallskip

Now we consider the  cases that follow, taking into account if $\Delta$ is zero or not.
\medskip

\noindent
{\bf Case 3b (i).}  If $\Delta=0$ there is no solution neither for \eqref{Donsimon} nor \eqref{Donsimondos}. Hence the Leavitt path algebra has IBN and it will be studied later.
\medskip

\noindent
{\bf Case 3b (ii).}   If $\Delta\ne 0$, using \eqref{Donsimon} and \eqref{Donsimondos} we get $$\frac{\Delta k_2}{l_1-1-t_1}=m-1=
\frac{\Delta k_1}{l_2-1-t_2}.$$
\noindent Consequently we get the equation $k_1 (l_1-1-t_1)=k_2(l_2-1-t_2)$, which has solutions. The minimum values for $k_1$ and $k_2$ are
$$\displaystyle\frac{l_2-1-t_2}{\gcd(l_1-1-t_1,l_2-1-t_2)} \text { and }  \displaystyle\frac{l_1-1-t_1}{\gcd(l_1-1-t_1,l_2-1-t_2)}, \text{ respectively. } $$ 
At this point, to make reference to the graph we are considering, we change the notation and take $\Delta_E=\Delta=(l_1-1)(l_2-1)-t_1t_2$ for simplicity,  
we have $$m=1+\frac{\vert\Delta_E\vert}{\gcd(l_1-1-t_1,l_2-1-t_2)}.$$


These computations are summarized in the result that follows.

\begin{lemma} \label{cava} Let $K$ be a field and $E$ be a graph of the form
 $$ 
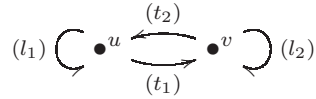
\xymatrix{
{\bullet}^u  \ar@(ul,dl)_{(l_1)} \ar@/^-.5pc/[r]_{(t_1)}  & {\bullet}^v
\ar@(ur,dr)^{(l_2)}    \ar@/_.5pc/[l]_{(t_2)}}
$$ 
where $(l_i,t_i)\ne (0,0)$ and $l_i -t_i \neq 1$, for $i = 1,2$. 
Then $L_K(E)$ does not have IBN  if and only if 
$\Delta_E \ne 0$. Moreover,  $L_K(E)$ has type  $\Big(1,1+\displaystyle\frac{\vert\Delta_E\vert}{\gcd(l_1-1-t_1,l_2-1-t_2)}\Big)$. 
\end{lemma}

We collect all the information of Lemmas \ref{oporto}, \ref{ribeiro}, \ref{madeira}, \ref{albarino} and \ref{cava} in Figure \ref{tableone}. To simplify, we associate  the set $$S=\{(l_1,t_1),(l_2,t_2)\}$$
\noindent to any two-vertex Leavitt path algebra.  
All possible  Leavitt path algebras not having IBN are the ones whose associated sets $S$ are listed below:
{\small \begin{figure}[H]
\begin{tabular}{|c|c|c|}
\hline
& $S$ & Type $(1,k)$\\
\hline
I &$\Big\{(0,0), (t_2+1,t_2)\ \vert\  t_2\ge 1\Big\}$ & $k=1+t_2$\\
II &$\Big\{(1,0), (t_2+1,t_2)\ \vert\ t_2\ge 1\Big\}$ & $k=1+t_2$\\
III &$\Big\{(t_1+1,t_1),(t_2+1,t_2)\ \vert\  t_1, t_2\ge 1\Big\}$ & $k=1+\gcd(t_1,t_2)$\\
IV &$\Big\{(t_1+1,t_1),(l_2,t_2)\ \vert\ (l_2, t_2)\neq (0,0), l_2-t_2\ne 1, t_1\ge 1\Big\}$ & $k=1+t_1$\\
V &$\Big\{(l_1,t_1),(l_2,t_2)\ \vert l_i-t_i\ne 1, \Delta_E\ne 0, \text{ for } i=1,2\Big\}$ & $k=1+\displaystyle\frac{\vert\Delta_E\vert}{\gcd(l_1-1-t_1,l_2-1-t_2)}$ \\
\hline
\end{tabular}
\medskip
\caption{Invariants (Part I) for two-vertex Leavitt path algebras not having IBN}\label{tableone}
\end{figure}}

%
%

\section{Computation of  $K_0(L_K(E))$}

In the previous section we have computed the type of the two-vertex Leavitt path algebras which do not have IBN. The type is not the only invariant that we must use in order to classify those algebras. This is why we compute here  $K_0(L_K(E))$ in the cases that appear in Figure \ref{tableone}. We will remark that the order of the order unit is related to the type.

 Recall that $N'_E:=A_E-1'$, where $1'$ denotes the matrix built from the identity matrix changing the columns corresponding to sinks by columns of zeros.
\medskip

\noindent
{\bf Case I.} We have $S= \Big\{(0,0), (t_2+1,t_2)\ \vert\  t_2\ge 1\Big\}$, whose associated graph $E$ is
$$ 
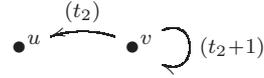
\xymatrix{
{\bullet}^u    & {\bullet}^v
\ar@(ur,dr)^{(t_2+1)}    \ar@/_.5pc/[l]_{(t_2)}}
$$
\medskip

Then $A_E = \begin{pmatrix} 0 & 0 \\ t_2 & t_2+1\end{pmatrix}$ and 
$N'_E=\begin{pmatrix}0 & 0\\ t_2  &t_2 \end{pmatrix}$.
By \cite{HH}, the Smith normal form of $N_E'$ is $ \begin{pmatrix}0 & 0\\ 0 &t_2 \end{pmatrix}$. This implies (as follows by \cite[Theorem 4.2]{AB}) that  $K_0(L_K(E))$ is isomorphic to $\Z \times \Z_{t_2}$.

\medskip

\noindent
{\bf Case II.} Now $S= \Big\{(1,0), (t_2+1,t_2)\ \vert\  t_2\ge 1\Big\}$ and its associated graph $E$ is
$$ 
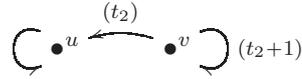
\xymatrix{
{\bullet}^u  \ar@(ul,dl)   & {\bullet}^v
\ar@(ur,dr)^{(t_2+1)}    \ar@/_.5pc/[l]_{(t_2)}}$$
\medskip

We have $A_E = \begin{pmatrix}1 & 0\\ t_2 & t_2+1 \end{pmatrix}$ and $N'_E=\begin{pmatrix}0 & 0\\ t_2 & t_2
\end{pmatrix}$. Again, the Smith normal form of $N_E'$ is $ \begin{pmatrix}0 & 0\\ 0 &t_2 \end{pmatrix}$ and  $K_0(L_K(E)) $ is isomorphic to $\Z \times \Z_{t_2}$.
\medskip

\noindent
{\bf Case III.} We have $S=\Big\{(t_1+1,t_1),(t_2+1,t_2)\ \vert\  t_1, t_2\ge 1\Big\}$, whose associated graph $E$ is
$$ 
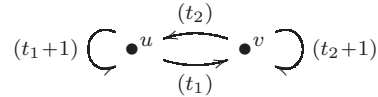
\xymatrix{
{\bullet}^u  \ar@(ul,dl)_{(t_1+1)} \ar@/^-.5pc/[r]_{(t_1)}   & {\bullet}^v
\ar@(ur,dr)^{(t_2+1)}    \ar@/_.5pc/[l]_{(t_2)}}
$$
\medskip

Then, $A_E = \begin{pmatrix}t_1+1 & t_1\\ t_2 & t_2+1 \end{pmatrix}$ and $N'_E=\begin{pmatrix}t_1 & t_1\\ t_2 & t_2
\end{pmatrix}$. The Smith normal form of $N_E'$ is $\begin{pmatrix} d & 0\\ 0 & 0 \end{pmatrix}$, where $d= \gcd(t_1, t_2)$.  Therefore $K_0(L_K(E))$ is isomorphic to $\Z \times \Z_d$.
\medskip


\noindent
{\bf Case IV.} Here $S=\Big\{(t_1+1,t_1),(l_2,t_2)\ \vert\ (l_2, t_2)\neq (0,0), l_2-t_2\ne 1, t_1\ge 1\Big\}$ and its associated graph $E$ is
$$ 
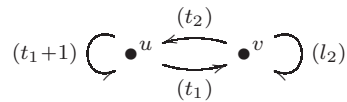
\xymatrix{
{\bullet}^u  \ar@(ul,dl)_{(t_1+1)} \ar@/^-.5pc/[r]_{(t_1)}   & {\bullet}^v
\ar@(ur,dr)^{(l_2)}    \ar@/_.5pc/[l]_{(t_2)}}
$$
\medskip

Then, $A_E = \begin{pmatrix}t_1+1 & t_1\\ t_2 & l_2 \end{pmatrix}$ and $N'_E=\begin{pmatrix}t_1 & t_1\\ t_2 & l_2-1
\end{pmatrix}$. In this case, the Smith normal form of $N_E'$ is $\begin{pmatrix} d & 0\\ 0 & \frac{\vert t_1(l_2-t_2-1)\vert}{d} \end{pmatrix}$, where $d= \gcd(t_1, t_2, l_2-1)$.  Therefore $K_0(L_K(E))$ is isomorphic to $\Z_d \times \Z_{\frac{\vert t_1(l_2-t_2-1)\vert}{d}}$.
\medskip

\noindent
{\bf Case V.}  For the final case, $S=\Big\{(l_1,t_1),(l_2,t_2)\ \vert\  \Delta_E\ne 0, l_i-t_i\ne 1, \text{ for } i=1,2\Big\}$ and the associated graph $E$ is as follows
$$ 
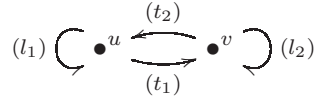
\xymatrix{
{\bullet}^u  \ar@(ul,dl)_{(l_1)} \ar@/^-.5pc/[r]_{(t_1)}   & {\bullet}^v
\ar@(ur,dr)^{(l_2)}    \ar@/_.5pc/[l]_{(t_2)}}
$$
\medskip

In this case, $A_E = \begin{pmatrix}l_1 & t_1\\ t_2 & l_2 \end{pmatrix}$ and $N'_E=\begin{pmatrix}l_1-1 & t_1\\ t_2 & l_2-1
\end{pmatrix}$. The Smith normal form of $N_E'$ is $\begin{pmatrix} d & 0\\ 0 & \frac{\vert \Delta_E\vert}{d} \end{pmatrix}$, where $d= \gcd(l_1-1, t_1, l_2-1, t_2)$.  Therefore $K_0(L_K(E))$ is isomorphic to $\Z_d \times \Z_{\frac{\vert\Delta_E \vert}{d}}$.
\medskip

We will write $d_E$ and $\Delta_E$ to refer to the greatest common divisor of the entries and to the determinant of $N'_E$, respectively. 

To end this section we remark that the order of $[1]_E$ in $K_0(L_K(E))$ is $n$, where $(1,1+n)$ is the type of $L_K(E)$. To prove this, take into account the monomorphism from the Grothendieck group of \ $\mathcal{U}(L_K(E))$ into $ K_0(L_K(E))$ given in \cite[Proposition 2.5]{KO}, together with the
fact that $L_K(E)^{n+1}\cong L_K(E)$ (as $L_K(E)$-modules). Observe that  in  Case IV, $t_1$ divides $\frac{\Delta_E}{d_E}$ and in Case V,  we have that
$\displaystyle \frac{\vert\Delta_E\vert}{\gcd(l_1-1-t_1,l_2-1-t_2)} \text{ divides } \frac{\vert \Delta_E\vert}{d_E}$, as expected from
\cite[Proposition 2.5]{KO}.


In Figure \ref{tabletwo}, that summarizes the computation of the $K_0$ groups, we note that $K_0(L_K(E))$ is of the form $\Z_{\frac{\vert\Delta_E\vert}{d_E}}\times\Z_ {d_E}$ in each of the cases. 

{\begin{figure}[H] 
\resizebox{17cm}{!}{
\begin{tabular}{|c|c|c|c|c|}
\hline
& $S$ & Type $(1,k)$ & $\Delta_E$&$K_0(L_K(E))$\\
\hline
 I &$\Big\{(0,0), (t_2+1,t_2)\ \vert\  t_2\ge 1\Big\}$ & $k=1+d_E$ & $0$ &$ \Z_{d_E}\times\Z$ \\
II & $\Big\{(1,0), (t_2+1,t_2)\ \vert\ t_2\ge 1\Big\}$ & $k=1+d_E$ &  $0$ & $ \Z_{d_E}\times\Z$ \\
III &$\Big\{(t_1+1,t_1),(t_2+1,t_2)\ \vert\  t_1, t_2\ge 1\Big\}$ & $k=1+d_E$ &$0$ & $ \Z_{d_E}\times\Z$\\
IV & $\Big\{(t_1+1,t_1),(l_2,t_2)\ \vert\ (l_2, t_2)\neq (0,0), l_2-t_2\ne 1, t_1\ge 1\Big\}$ & $k=1+t_1$ & $\ne 0$ &$\Z_{d_E} \times \Z_{\frac{\vert \Delta_E\vert}{d_E}}$ \\
 V & $\Big\{(l_1,t_1),(l_2,t_2)\ \vert\ l_i-t_i\ne 1, \Delta_E\ne 0, \text{ for } i=1,2\Big\}$ & $k=1+\displaystyle\frac{\vert\Delta_E\vert}{\gcd(l_1-1-t_1,l_2-1-t_2)}$ & $\ne 0$ & $\Z_{d_E} \times \Z_{\frac{\vert \Delta_E\vert}{d_E}}$ \\
\hline
\end{tabular}}
\medskip
\caption{Invariants (Part II) for two-vertex  Leavitt path algebras not having IBN}\label{tabletwo}
\end{figure}}

\begin{remark}\label{jumilla}
\rm Any  Leavitt path algebra associated to a graph in Cases I, II, III is not isomorphic to any Leavitt path algebra in Cases  IV or V. 
Indeed, notice that in Cases I, II and III in Figure \ref{tabletwo}, $\Delta_E=0$, hence the $K_0$ groups are isomorphic to $\Z_{d_E} \times \Z$. So, in these cases the $K_0$ group has a torsion-free part while in 
Cases IV and V the $K_0$ is a torsion group. 
\end{remark}

\section{Classification of Leavitt path algebras not having IBN}  
In this section we  study the isomorphisms between the algebras in the different cases in Figure \ref{tabletwo} following the decision tree.  
Recall that $P_l$, $P_c$ and $P_{ec}$ denote the set of all line points, vertices in cycles with no exits and vertices in extreme cycles of $E$, respectively.
We will compute the ideals generated by the above sets, namely, $I(P_c)$, $ I(P_l)=\soc(L_K(E))$ and $I(P_{ec})$. Clearly, the socle is  invariant under isomorphisms and the ideal $I(P_c)$ is proved to be also invariant under isomorphisms  in \cite{ABS}. 

We start  by proving that $I(P_{ec})$ remains invariant under ring isomorphisms when $E$ is a finite graph.

\begin{theorem}
\label{EnateMerlot-Merlot} Let $E$ be a finite graph. Then the ideal $I(P_{ec})$ is invariant under any ring isomorphism. 
\end{theorem} 

\begin{proof} 
Assume that $E$ and $F$ are finite graphs and that $\varphi: L_K(E) \to L_K(F)$  is a ring isomorphism. 

Denote by $P_{ec}^E$ and  $P_{ec}^F$ the hereditary subsets of $E^0$ and $F^0$, respectively, consisting of vertices in extreme cycles in $E$ and $F$, respectively.

As any isomorphism sends idempotents to idempotents and by \cite[Corollary 2.9.11]{AAS}, $\varphi(I(P^E_{ec}))$ is a graded ideal (at a first glance it is a graded ring ideal but, taking into account \cite[Remark 1.2.11]{AAS} it is actually a graded algebra ideal). Hence, $\varphi(I(P^E_{ec}))= I(H)$ for some hereditary saturated set $H$ in $F$. Moreover, $H = F^0 \cap \varphi(I(P^E_{ec}))$ by \cite[Theorem 2.4.8]{AAS}.

Take $v \in H$. Since the graph is finite, $v$ has to connect to a line point, to a cycle without exits, or to an extreme cycle (by \cite[Theorem 2.9 (ii)]{CMMSS2} the ideal of $L_K(F)$ generated by the hereditary set $P^F_{lec}$ consisting of line points, vertices in cycles without exits and vertices in extreme cycles is dense and by \cite[Propostion 1.10]{CMMSS2} we have that $I(P^F_{lec})$ is dense if and only if every vertex of the graph connects to a vertex in $P^F_{lec}$). We are going to prove that the only option for $v$ is to connect to an extreme cycle. Assume that $v$ connects to a line point, say $w$, or to a cycle without exits, say $c$. Since $H$ is hereditary, $w\in H$ or $c^0\subseteq H$. This implies that $I(H)$ contains a primitive idempotent (see \cite[Proposition 5.3]{ABS}). Since primitive idempotents are preserved by isomorphisms, this means that $I(P^E_{ec})$ contains a primitive idempotent. But this is a contradiction because of \cite[Corollary 4.10]{ABS}, since $I(P^E_{ec})$ is purely infinite simple (by \cite[Proposition 2.6]{CMMSS2}). Applying the (CK2)  relation, $v$ is in the ideal  $I(P_{ec}^F)$; therefore, $\varphi(I(P^E_{ec})) \subseteq I(P_{ec}^F)$. Reasoning in the same way with $\varphi^{-1}$ we get $\varphi^{-1}(I(P^F_{ec})) \subseteq I(P^E_{ec})$, implying $\varphi(I(P^E_{ec})) = I(P^F_{ec})$.
\end{proof}

In the proceeding study we will follow the steps indicated in Figure \ref{WeAreHappy}. Note that the three sets $P_l$, $P_{c}$, $P_{ec}$ will play an important role in the classification of  two-vertex Leavitt path algebras not having IBN. We start by considering the different possibilities for the socle (non-zero or zero).

\subsection{$\soc(L_K(E))\ne 0$.}
\noindent

\medskip
From now on we will denote by $\mathcal C$ the class of Leavitt path algebras with nonzero socle, not having IBN which are associated to two-vertex graphs.

Summarizing the information contained in the one-sink case (Subsection \ref{OneSinkCase}) we get the lemma that follows, where $\rm{Type}(X)$ denotes the type of $X$ (in the sense of \cite[Definition 1.1.1]{AAS}).

\begin{lemma}
For every  $A\in \mathcal{C}$ the associated two-vertex graph, say $E_{l_2}$, is
$$\xymatrix{
{\bullet}^u   & {\bullet}^v
\ar@(ur,dr)^{(l_2)}    \ar@/_.5pc/[l]_{(l_2-1)}}
$$
with $l_2>1$. Then, $\soc(A)\cong M_{\infty}(K)$, $\bar{A}={A/\soc(A)}\cong L(1,l_2)$ and the type of $A$ is $(1,l_2)$. A complete system of invariants for $\mathcal C$ is the type. Also the quotient algebra is an invariant for $\mathcal C$. More pecisely, for two algebras $A$ and $B$ in $\mathcal C$ the following assertions are equivalent:
\begin{enumerate}[\rm (i)]
\item $A \cong B$.
\item $\bar A \cong \bar B$.
\item $\rm{Type}(A)=\rm{Type}(B)$.
\end{enumerate}
\end{lemma}
\begin{proof} Since $\soc(A)\ne 0$ there must be line-points in $E_{l_2}$, so this graph contains sinks. The unique graphs of this type which produce a Leavitt path algebra not having IBN are given in Lemma \ref{oporto}. Assume $A,B\in\mathcal{C}$. If $A\cong B$ then $\bar A\cong \bar B$ since every isomorphism preserves the socle.
Now, if $\bar A\cong \bar B$, $A=L_K(E_{l_2})$ and $B=L_K(E_{m_2})$, then we have $\bar A\cong L(1,l_2)$ and
$\bar B=L(1,m_2)$, hence $(1,l_2)=\hbox{Type}(L(1,l_2))=\hbox{Type}(L(1,m_2))=(1,m_2)$ giving $l_2=m_2$ hence the underlying graphs are the same and so
$A =B$. Finally, if $\text{Type}(A)=\text{Type}(B)$, then, by Lemma \ref{oporto},  $l_2=m_2$ so that $A \cong B$.
\end{proof}
 
 \subsection{$\soc(A)=0$.}
 \noindent
 
 \medskip
We  focus our attention on algebras $A=L_K(E)$ with $\vert E^0\vert=2$ and $\soc(A)=0$. We also rule out the purely infinite simple case because for this class a  system of invariants is well known: the Franks triple $(K_0,[1]_E,\vert N'_E\vert)$ (see \cite[Corollary 2.7]{Flow}).

\subsubsection{Case 1. Decomposable algebras}
\medskip
Denote by $\mathcal D$ the class of decomposable Leavitt path algebras  with zero socle, not having IBN and such that their associated graphs have two vertices.

\begin{lemma}\label{V(a)}
For every $A\in\mathcal{D}$ the associated two-vertex graph, say  $E_{l_1,l_2}$, is of the form
$$ \xymatrix{
{\bullet}^u\ar@(ul,dl)_{(l_1)}    & {\bullet}^v\ar@(ur,dr)^{(l_2)}     }
$$
\smallskip

\noindent with $l_1,l_2>1$. Given  two graphs $E_{l_1,l_2}$ and $E_{n_1,n_2}$,
we have $L_K(E_{l_1,l_2})\cong L_K(E_{n_1,n_2})$ if and only if $(l_1,l_2)=(n_1,n_2)$ or $(n_2,n_1)$. 
Thus, a complete system of invariants for the algebras $A$ in $\mathcal D$ is the pair $(l_1,l_2)$, where $l_1,l_2$ are the types of the unique two graded ideals of $A$, considered as algebras.
\end{lemma}
\begin{proof}
Take $A \in \mathcal D$ and let $E$ be the two-vertex graph associated to $A$. By \cite[Theorem 6.5]{CMMS} there  must be hereditary saturated  nonempty subsets $H_1, H_2$ whose union is $E^0=\{u, v\}$. Then, necessarily $H_1=\{u\}$ and $H_2=\{v\}$, so there is no edge connecting $u$ to $v$ and vice versa. Thus, $E$ consists of $l_1$ loops  based at $u$ and $l_2$ loops  based at $v$. 
Taking into account Case 3.b in Subsection \ref{Eristoff}, the necessary and sufficient conditions for $A$ not to have IBN are $l_1,l_2>1$. Suppose $A=L_K(E_{l_1,l_2})\cong L(1,l_1)\oplus L(1,l_2)$ and $B=L_K(E_{n_1,n_2})\cong L(1,n_1)\oplus L(1,n_2)$ are in $\mathcal D$. The unique proper non-zero graded ideals in $A$ are isomorphic to $L(1,l_1)$ and $L(1,l_2)$, while  for $B$ they are isomorphic to $L(1,n_1)$ and $L(1,n_2)$. By \cite[Proof of Theorem 5.3]{AMP} an ideal is graded if and only if it is generated by idempotents, therefore graded ideals are preserved by isomorphism. Thus, if $A$ and $B$ are isomorphic, then $(l_1,l_2)=(n_1,n_2)$ or $(n_2,n_1)$. 
\end{proof}

\subsubsection{Case 2. Indecomposable algebras.}

 Let $A=L_K(E)$ be an indecomposable Leavitt path algebra, where $E$ is a two-vertex graph.
  Then $A$ can be either purely infinite simple or not.
\smallskip

{\it Case 2.1. \underline{Purely infinite simple algebras}}.
For these algebras we can use \cite[Corollary 2.7]{Flow}
which provides an invariant, the Franks triple, for determining the isomorphism classes.

 \begin{remark}\label{cacique}
 \rm
Assume  $E$ is a finite graph. Then $A:=L_K(E)$ is purely infinite simple if and only if $A$ is simple and $\soc(A)=0$
(see \cite[Corolario 3.4.3]{Maria}). 
 \end{remark}
 
  By  \cite[Lemma 2.7]{CMMSS2} the sum of the ideals $I(P_l)$, $I(P_c)$ and $I(P_{ec})$ is direct. In fact, since $E$ is finite, then this sum is a dense ideal in $A$ by 
 \cite[Theorem 2.9]{CMMSS2}.  Since $A$ is simple  then $A$ must coincide with $I(P_l)$, with $I(P_c)$ or with
 $I(P_{ec})$. 
 
Using Remark \ref{cacique} we get that $I(P_l)=0$; also $I(P_c)=0$ because the algebra is purely infinite simple; therefore, necessarily $A=I(P_{ec})$.
 
 \smallskip
 
{\it Case 2.2. \underline{Non-purely infinite simple algebras.}}

By Remark \ref{cacique},  the non-purely infinite simple Leavitt path algebras in this case  are non-simple. 
 \smallskip

{\it Case 2.2.1. $I(P_c)\ne 0$.}  Notice that there are no sinks and there is a cycle with no exits. The possibilities are:
$$
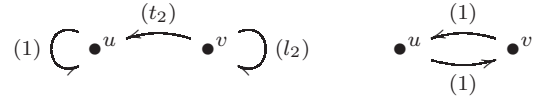
\xymatrix{
{\bullet}^u\ar@(ul,dl)_{(1)}   & {\bullet}^v
\ar@(ur,dr)^{(l_2)}    \ar@/_.5pc/[l]_{(t_2)}}
\qquad 
\xymatrix{
{\bullet}^u\ar@/_.5pc/[r]_{(1)}   & {\bullet}^v
    \ar@/_.5pc/[l]_{(1)}}
$$
The Leavitt path algebra associated to the first graph does not have IBN only if $t_2=l_2-1$ (see Lemma \ref{ribeiro}), the type of the corresponding  Leavitt path algebra is $(1,l_2)$. The Leavitt path algebra associated to the second graph has IBN,
concretely it is isomorphic to $M_2(K[x, x^{-1}])$ by \cite[Theorem 3.3]{AAS2}.
So the type is again a sufficient invariant
to determine the isomorphism classes in this case. 
\smallskip

{\it Case 2.2.2. $I(P_c)=0$.}
By \cite[Proposition 1.10 and Theorem 2.9 (ii)]{CMMSS2} we have that in a finite graph any vertex connects either to a sink or to a cycle without exits or to an extreme cycle. Since there are no sinks and no cycles without exits, any vertex connects to an extreme cycle. Hence $I(P_{ec})\ne 0$. Moreover, as $I(P_{ec})$ is purely infinite simple (see  \cite[Proposition 2.6]{CMMSS2}), it has to be a  proper ideal of $L_K(E)$. We see that the only possible graph
is of the form:

$$
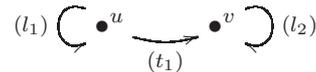
\xymatrix{
{\bullet}^u\ar@(ul,dl)_{(l_1)}  \ar@/_.5pc/[r]_{(t_1)} & {\bullet}^v
\ar@(ur,dr)^{(l_2)}    }
$$
with $l_2>1$, as any cycle should have an exit. Also
 $t_1\ge 1$,  otherwise the algebra would be decomposable. Moreover, $l_1\ge 1$ because if $l_1=0$ then the Leavitt path algebra would be simple (by \cite[3.11 Theorem]{AA1}), a contradiction as we are assuming that
 the algebra is non-simple.
 Furthermore, if $l_1=1$ the algebra would have IBN by Lemma \ref{cava} because $\Delta_E=0$ in this case, so we have to assume $l_1\ge 2$.
 \smallskip
 
 To  convince the reader that we have completed the decision tree we just point out that $P_{ec}\neq \emptyset$ because every vertex connects to one vertex in $P_l\sqcup P_c\sqcup P_{ec}$ and, in our case, only $P_{ec}$ survives.
 
 \smallskip

We summarize all the data of this section in  Figure \ref{tablethree}. The order in which the graphs appear in the table below corresponds to the order in which the cases have been studied in the decision tree.
\begin{figure}[H]
\resizebox{17cm}{!}{
\begin{tabular}{|c|c|c|cc|c|c|c|c|}
\hline
& Decomposable & PIS & Graph  & & $I(P_l)$ &  $I(P_c)$ & $I(P_{ec})$ & $L_K(E)/I(P_{lec})$ \\
\hline
I &  NO  & NO & \hskip 1.2cm $$\xymatrix{
{\bullet}^u   & {\bullet}^v
\ar@(ur,dr)^{(t_2+1)}    \ar@/_.5pc/[l]_{(t_2)}}
$$ & with  $t_2\ge 1$ & $ \M_\infty(K)$ & $0$ & $0$ &$L(1,t_2+1)$  \\
& & & & & & & & \\
V(a)&  YES & NO &  $$ \xymatrix{
{\bullet}^u\ar@(ul,dl)_{(l_1)}    & {\bullet}^v\ar@(ur,dr)^{(l_2)}     }
$$  & $ l_1, l_2\ge 2$ & 0 & 0 & $  \bigoplus\limits_{i=1}^2 L(1,l_i) $ &$0$\\
& & & & & & & & \\
III &  NO & YES  & $$ \xymatrix{
{\bullet}^u  \ar@(ul,dl)_{(t_1+1)} \ar@/^-.5pc/[r]_{(t_1)}   & {\bullet}^v
\ar@(ur,dr)^{(t_2+1)}    \ar@/_.5pc/[l]_{(t_2)}}
$$ 
& $t_1, t_2\ge 1$ & $0$ &$0$ & $L_K(E)$ & $0$\\
& & & & & & & & \\
IV(a) &  NO & YES  & $$ \hskip -.3cm\xymatrix{
{\bullet}^u  \ar@(ul,dl)_{(t_1+1)} \ar@/^-.5pc/[r]_{(t_1)}   & {\bullet}^v
\ar@(ur,dr)^{(l_2)}    \ar@/_.5pc/[l]_{(t_2)}}
$$ & $l_2-t_2\ne 1; t_1,t_2\ge 1$ & $0$ & $0$ & $L_K(E)$ & $0$\\
& & & & & & & & \\
V(c) &  NO & YES  & $$ \xymatrix{
{\bullet}^u  \ar@/^-.5pc/[r]_{(t_1)}   & {\bullet}^v
  \ar@/_.5pc/[l]_{(t_2)}}
$$ & $t_1,t_2\ne 0$; $t_1\ge 2 \text{ or } t_2\ge 2$
 & $0$ & $ 0$ & $L_K(E)$ &$0$\\
V(d) &  NO & YES   & $$ \hskip -1cm \xymatrix{
{\bullet}^u  \ar@(ul,dl)_{(l_1)} \ar@/^-.5pc/[r]_{(t_1)}   & {\bullet}^v
 \ar@/_.5pc/[l]_{(t_2)}}
$$ & $l_1,t_1,t_2\ge 1; l_1-t_1\ne 1$ & $0$ & $ 0$ & $L_K(E)$ &$0$\\
V(e)  &  NO & YES  & $$ \hskip 1cm\xymatrix{
{\bullet}^u   \ar@/^-.5pc/[r]_{(t_1)}   & {\bullet}^v
\ar@(ur,dr)^{(l_2)}   }
$$ & $t_1\ge 1; l_2\ge 2$ & $0$ & $ 0$ & $ \M_{t_1+1}(L_K(1,l_2))$ &$0$\\
V(f) &  NO & YES    & 
$$ \xymatrix{
{\bullet}^u  \ar@(ul,dl)_{(l_1)} \ar@/^-.5pc/[r]_{(t_1)}   & {\bullet}^v
\ar@(ur,dr)^{(l_2)}    \ar@/_.5pc/[l]_{(t_2)}}
$$ &
$\Delta_E\ne 0; l_i, t_i\ge 1; l_i-t_i\ne 1, \text{ for  } i=1,2$ & $0$ & $ 0$ & $L_K(E)$ &$0$\\
& & & & & & & & \\
II &  NO & NO    & \hskip 0.7cm $$\xymatrix{
{\bullet}^u\ar@(ul,dl)_{(1)}   & {\bullet}^v
\ar@(ur,dr)^{(t_2+1)}    \ar@/_.5pc/[l]_{(t_2)}}$$ & with $t_2\ge 1$ & $0$ &  $\M_\infty(K[x,x^{-1}])$ &   $0$& $L(1,t_2+1)$\\ 
& & & & & & & & \\
IV(b)-V(b) &  NO & NO  &  \hskip .2cm $$\xymatrix{
{\bullet}^u\ar@(ul,dl)_{(l_1)}  \ar@/_.5pc/[r]_{(t_1)} & {\bullet}^v
\ar@(ur,dr)^{(l_2)}    }
$$ &
with $l_1,l_2 \ge 2$; $t_1\ge 1$ & $0$ & $0$  & $\M_{\infty}(L_K(1,l_2))$ & $L(1, l_1)$\\ 
& & & & & & & & \\
\hline
\end{tabular}}
\medskip
\caption{Invariants (Part III) for two-vertex Leavitt path algebras not having IBN}\label{tablethree}
\end{figure}

The cases appearing in this table follow from the cases in Figure \ref{tabletwo}. 
We call Case IV (a) to Case IV in Figure  \ref{tabletwo} for $t_2\neq 0$, Case IV (b) is  Case IV in Figure  \ref{tabletwo} for $t_2= 0$, and Case V(b) is Case V  in Figure  \ref{tabletwo} for $t_2=0, l_1, l_2 \geq 2$.




We justify that Cases IV(b) and V(b) are isomorphic. Any graph $E$ in Case IV(b) is as follows:  
$$
\xymatrix{
{\bullet}^u  \ar@(ul,dl)_{(t_1+1)} \ar@/^-.5pc/[r]_{(t_1)}   & {\bullet}^v
\ar@(ur,dr)^{(l_2)}  } 
$$
where $l_2,t_1 \in \N$, $l_2 \ge 2$. 
Now, consider the graph $F$ 
$$
\xymatrix{
{\bullet}^u  \ar@(ul,dl)_{(t_1+1)} \ar@/^-.5pc/[r]_{(s)}   & {\bullet}^v
\ar@(ur,dr)^{(l_2)}  } 
$$
where $s = (l_2-1) + t_1$ and $\Delta_{F}= t_1 (l_2 -1) \neq 0$, which is  in Case V(b). Note that $E$ is produced by a shift move from $F$ and by Theorem \ref{LuisFelipe} the Leavitt path algebras $L_K(E)$ and $L_K(F)$ are isomorphic. Therefore, it is enough to find the isomorphism classes in Case V(b). 
\medskip

In what follows we are going to compare Cases V(c) and V(d) to V(e).  

Take any graph from V(c), $$ \xymatrix{ E : & &
{\bullet}^u   \ar@/^-.5pc/[r]_{(t_1)}   & {\bullet}^v
  \ar@/_.5pc/[l]_{(t_2)}}
$$
where $t_1 \geq 2$ or $t_2 \geq 2$ and without loss of generality 
$t_1 \geq t_2$.

Consider the  graph 
$$ \xymatrix{ F : & &
{\bullet}^u   \ar@/^-.5pc/[r]_{(t_1)}   & {\bullet}^v
\ar@(ur,dr)^{(t_1t_2)}   }
$$
which is in V(e). It can be transformed into the graph $E$ via consecutive 
$t_2$-many shift moves  of $s^{-1}(u) \to s^{-1}(v)$. 
{By Theorem \ref{LuisFelipe}, the Leavitt path algebras $L_K(E)$ and $L_K(F)$ are isomorphic.}
\bigskip

Take any graph from V(d), for example:
$$ \xymatrix{ E : & &
{\bullet}^u   \ar@/^-.5pc/[r]_{(t_1)}   & {\bullet}^v
\ar@(ur,dr)^{(l_2)}    \ar@/_.5pc/[l]_{(t_2)}}
$$ where $l_2 - t_2 \neq 1$.

The  graph 
$$ \xymatrix{ F: & &
{\bullet}^u   \ar@/^-.5pc/[r]_{(t_1)}   & {\bullet}^v
\ar@(ur,dr)^{(l_2+t_1t_2)}   }
$$
which is in V(e), can be transformed into the graph $E$ via consecutive $t_2$-many shift moves of 
$s^{-1}(u) \to s^{-1}(v)$. 
{Again, by Theorem \ref{LuisFelipe}, the Leavitt path algebras $L_K(E)$ and $L_K(F)$ are isomorphic.} Thus, in Figure \ref{tablethree} we may eliminate the rows corresponding to the cases V(c) and V(d).
\bigskip 

Any graph from V(e) produces a Leavitt path algebra isomorphic to $M_{t_1+1}(L(1,l_2))$. 
Take a graph $E$ in  Case V(e):  
$$ \xymatrix{
{\bullet}^u  \ar@/^-.5pc/[r]_{(t_1)}   & {\bullet}^v
\ar@(ur,dr)^{(l_2)}    }
$$
where $t_1 \geq 1, l_2 \geq 2$. 
Since $E$ is an $l_2$-rose comet, by \cite[Proposition 2.2.19 ]{AAS}
$L_K(E) \cong  M_{t_1+1}(L(1,l_2))$. 
\medskip

The previous reasoning, as well as the table that follows will allow to refine Figure \ref{tablethree}. 

Recall that the \emph{Betti number} of a finitely generated abelian group  $G$, denoted by ${\rm B}(G)$ is the dimension (as a $\Z$-module) of the free part of $G$.

In Figure \ref{abanicoMuge},  an entry 1 or  0 in the first and the second columns will mean that $P_l(E)$ and  $P_c(E)$ are non-empty or empty, respectively. In the third column an entry 1 will mean that the Leavitt path algebra is decomposable, while 0 will stand for the opposite.  An entry 1 in the PIS column stands for a Leavitt path algebra which is purely infinite simple. An entry 1 in the $B(K_0)$ column represents the Betti number of the $K_0$ of the corresponding Leavitt path algebra. 

\begin{figure}[H]
\begin{center}
\begin{tabular}{|c|c|c|c||c|c|}
\hline
$P_l$ & $P_c$ & Dec & PIS & $\hbox{B}(K_0)$ & \cr
\hline
 $1$  &  $0$ & $0$ & $0$  &  $1$ & I \cr 
  $0$  &  $1$ & $0$ & $0$  &  $1$ & II \cr 
  $0$  &  $0$ & $1$ & $0$  &  $0$ & V(a) \cr 
  $0$  &  $0$ & $0$ & $1$  &  $1$ & III\cr
 $0$  &  $0$ & $0$ & $1$  &  $0$ &  IV(a),\ V(e),\ V(f)\cr 
 $0$  &  $0$ & $0$ & $0$  &  $0$ & V(b)\cr 
 \hline
\end{tabular}
\end{center}
\medskip
\caption{ Non-IBN cases}\label{abanicoMuge}
\end{figure}

\begin{figure}[H]
\resizebox{17cm}{!}{
\begin{tabular}{|c|c|c|cc|c|c|c|c|}
\hline
& Decomposable & PIS & Graph  & & $I(P_l)$ &  $I(P_c)$ & $I(P_{ec})$ & $L_K(E)/I(P_{lec})$ \\
\hline
I &  NO  & NO & \hskip 1.2cm $$\xymatrix{
{\bullet}^u   & {\bullet}^v
\ar@(ur,dr)^{(t_2+1)}    \ar@/_.5pc/[l]_{(t_2)}}
$$ & with  $t_2\ge 1$ & $ \M_\infty(K)$ & $0$ & $0$ &$L(1,t_2+1)$  \\
& & & & & & & & \\
V(a)&  YES & NO &  $$ \xymatrix{
{\bullet}^u\ar@(ul,dl)_{(l_1)}    & {\bullet}^v\ar@(ur,dr)^{(l_2)}     }
$$  & $ l_1, l_2\ge 2$ & 0 & 0 & $  \oplus_{i=1}^2 L(1,l_i) $ &$0$\\
& & & & & & & & \\
III &  NO & YES  & $$ \xymatrix{
{\bullet}^u  \ar@(ul,dl)_{(t_1+1)} \ar@/^-.5pc/[r]_{(t_1)}   & {\bullet}^v
\ar@(ur,dr)^{(t_2+1)}    \ar@/_.5pc/[l]_{(t_2)}}
$$ 
& $t_1, t_2\ge 1$ & $0$ &$0$ & $L_K(E)$ & $0$\\
& & & & & & & & \\
IV(a) &  NO & YES  & $$ \xymatrix{
{\bullet}^u  \ar@(ul,dl)_{(t_1+1)} \ar@/^-.5pc/[r]_{(t_1)}   & {\bullet}^v
\ar@(ur,dr)^{(l_2)}    \ar@/_.5pc/[l]_{(t_2)}}
$$ & $l_2-t_2\ne 1; t_1,t_2\ge 1$ & $0$ & $0$ & $L_K(E)$ & $0$\\
& & & & & & & & \\

V(e)  &  NO & YES  & $$ \xymatrix{
{\bullet}^u   \ar@/^-.5pc/[r]_{(t_1)}   & {\bullet}^v
\ar@(ur,dr)^{(l_2)}   }
$$ & $t_1\ge 1; l_2\ge 2$ & $0$ & $ 0$ & $ \M_{t_1+1}(L_K(1,l_2))$ &$0$\\
V(f) &  NO & YES    & 
$$ \xymatrix{
{\bullet}^u  \ar@(ul,dl)_{(l_1)} \ar@/^-.5pc/[r]_{(t_1)}   & {\bullet}^v
\ar@(ur,dr)^{(l_2)}    \ar@/_.5pc/[l]_{(t_2)}}
$$ &
$\Delta_E\ne 0; l_i, t_i\ge 1; l_i-t_i\ne 1, \text{ for  } i=1,2$ & $0$ & $ 0$ & $L_K(E)$ &$0$\\
& & & & & & & & \\
II &  NO & NO    & $$\xymatrix{
{\bullet}^u\ar@(ul,dl)_{(1)}   & {\bullet}^v
\ar@(ur,dr)^{(t_2+1)}    \ar@/_.5pc/[l]_{(t_2)}}$$ & with $t_2\ge 1$ & $0$ &  $\M_\infty(K[x,x^{-1}])$ &   $0$& $L(1,t_2+1)$\\ 
& & & & & & & & \\
V(b) &  NO & NO  &  $$\xymatrix{
{\bullet}^u\ar@(ul,dl)_{(l_1)}  \ar@/_.5pc/[r]_{(t_1)} & {\bullet}^v
\ar@(ur,dr)^{(l_2)}    }
$$ &
with $l_1,l_2 \ge 2$; $t_1\ge 1$; $l_1-t_1 \neq 1$ & $0$ & $0$  & $\M_{\infty}(L_K(1,l_2))$ & $L(1, l_1)$\\ 
& & & & & & & & \\
\hline
\end{tabular}}
\medskip
\caption{Invariants (Part III bis) for two-vertex Leavitt path algebras not having IBN}\label{tablethreebis}
\end{figure}

\begin{remark}\label{NonIso}
\rm
There is an overlap in Cases IV(a), V(e) and V(f). Consider, for example, the graphs

$$ 
E: \xymatrix{
{\bullet}^u  \ar@(ul,dl)_{(2)} \ar@/^-.5pc/[r]_{(2)}   & {\bullet}^v
\ar@(ur,dr)^{(1)}    \ar@/_.5pc/[l]_{(1)}}
\quad
F:  \xymatrix{
{\bullet}^u   \ar@/^-.5pc/[r]_{(2)}   & {\bullet}^v
\ar@(ur,dr)^{(3)}}
\quad
G: \xymatrix{
{\bullet}^u  \ar@(ul,dl)_{(3)} \ar@/^-.5pc/[r]_{(2)}   & {\bullet}^v
\ar@(ur,dr)^{(1)}    \ar@/_.5pc/[l]_{(1)}}
$$ 

We have that E, F and G are in cases IV(a), V(e) and V(f), respectively, and the corresponding Leavitt path algebras are isomorphic (via shift moves).
\end{remark}

We state the main result for Leavitt path algebras not having IBN under consideration. 

\begin{theorem}\label{lagardeisillaBis}
Let $E$ be a finite graph with two-vertices whose Leavitt path algebra $L_K(E)$ does not have IBN. Then, $L_K(E)$ is isomorphic to a Leavitt path algebra whose associated graph is
\begin{enumerate}[\rm (i)]
\item in Case I, if and only if $P_l(E)\neq \emptyset$ (which implies $P_c(E)=\emptyset$, $P_{ec}(E)=\emptyset$, $L_K(E)$  is neither decomposable nor purely infinite simple and $B(K_0(L_K(E)))= 1$). Any two Leavitt path algebras in this situation are isomorphic if and only if their types are the same.
\item  in Case II, if and only if $P_c(E)\neq \emptyset$ (which implies $P_l(E)=\emptyset$, $P_{ec}(E)=\emptyset$, $L_K(E)$  is neither decomposable nor purely infinite simple and $B(K_0(L_K(E)))= 1$). Any two Leavitt path algebras in this situation are isomorphic if and only if their types are the same.
\item in Case V(a) if and only if $L_K(E)$ is decomposable (which implies $P_l(E)=\emptyset$, $P_c(E)=\emptyset$, $P_{ec}(E)\neq\emptyset$, $L_K(E)$  is not purely infinite simple and $B(K_0(L_K(E)))= 0$). Any two Leavitt path algebras in this situation are isomorphic if and only if the sets of the types of the non-zero proper ideals coincide.
\item in Case III if and only if $L_K(E)$ is purely infinite simple and $B(K_0(L_K(E)))= 1$
 (which implies $P_l(E)=\emptyset$, $P_c(E)=\emptyset$, $P_{ec}(E)\neq\emptyset$). Any two Leavitt path algebras in this situation whose associated graphs have $S=\{(t_1+1, t_1), (t_2+1, t_2)\}$, $S'=\{(t'_1+1, t'_1), (t'_2+1, t'_2)\}$ are isomorphic if and only if   ${\rm gcd}(t_1, t_2) = {\rm gcd}(t'_1, t'_2) $.
 \item in Cases IV(a), V(e) or V(f) if and only if the Leavitt path algebra $L_K(E)$ is purely infinite simple and $B(K_0(L_K(E)))= 0$
 (which implies $P_l(E)=\emptyset$, $P_c(E)=\emptyset$, $P_{ec}(E)\neq\emptyset$). Any two Leavitt path algebras in this situation whose Franks triples coincide are isomorphic. On the other hand, if two Leavitt path algebras in these cases are isomorphic, then their Franks triples coincide up to the sign of the determinant.
 \item in Case V(b) if and only if $P_l(E)=\emptyset$, $P_c(E)= \emptyset$ (which implies $P_{ec}\neq \emptyset$),  $L_K(E)$   is neither decomposable nor purely infinite simple and $B(K_0(L_K(E)))=0$. Any two Leavitt path algebras in this situation whose associated graphs $E$ and $F$ are in Case V(b)  which are isomorphic must satisfy
$d_E = d_F$ and $gcd(l_1-1-t_1,l_2-1) = gcd(l_1-1-t_1',l_2-1)$.  
\end{enumerate}
\end{theorem}

\begin{proof}
By looking at Figures \ref{abanicoMuge} and \ref{tablethreebis}  we can distinguish  the different cases that appear in the statement.
\medskip

Now we study isomorphisms within each case. Consider a graph $E$ in either   Case I or Case II. 
Since $L_K(E)/I(P_{lec})$ is determined by
$t_2$, which is the only variable of the graph, each graph in these cases  produces a non-isomorphic Leavitt path algebra. 
Similarly, any graph in Case V(a) produces a distinct isomorphism class by Lemma \ref{V(a)}.
Let us study the graphs in Case III.  Consider  $E$ to be the graph with $S_E=\{(t_1+1,t_1),(t_2+1,t_2)\}$, and $F$ to be the graph with $S_F=\{(n+1,n),(n+1,n)\}$, where $n:=\gcd(t_1,t_2)$.
Then $K_0(L_K(E))$ and $K_0(L_K(F))$ are both isomorphic to $\Z\times\Z_n$ and  there is an isomorphism from $K_0(L_K(E))$ to $K_0(L_K(F))$ sending  $[1]_E$ to $[1]_F$, which are
both mapped to $(0,\bar 1)$ in $\Z\times \Z_n$. Moreover, the
determinants  agree: $\Delta_E=\Delta_F=0$. So, by \cite[Corollary 2.7]{Flow}, $L_K(E)$ is ring isomorphic to
$L_K(F)$.  Since the center of $L_K(E)$ is isomorphic to $K$ because the Leavitt path algebra is unital and purely infinite simple (see, for example \cite[Theorem 3.7]{CMMSS2} and \cite[Theorem 4.2]{AC}), we may apply Proposition \ref{RingVersusAlgebra} to get that there is an algebra isomorphism from $L_K(E)$ to $L_K(F)$.
Therefore, for any positive integer $n$, the graph with $S=\{(n+1,n),(n+1,n)\}$ produces an algebra isomorphism class.

\medskip
Consider any two graphs $E$, $F$ in Case V(b) such that $L_K(E) \cong L_K(F)$, where
$$
\xymatrix{ E: & &
{\bullet}^u  \ar@(ul,dl)_{(l_1)} \ar@/^-.5pc/[r]_{(t_1)}   & {\bullet}^v
\ar@(ur,dr)^{(l_2)}  } 
 \quad 
\xymatrix{  F: & &
{\bullet}^u  \ar@(ul,dl)_{(l_1')} \ar@/^-.5pc/[r]_{(t_1')}   & {\bullet}^v
\ar@(ur,dr)^{(l_2')}  } 
$$
with $t_1,t_1' \geq 1; l_1,l_2,l_2', l_1' \geq 2$.  
\medskip

Recall that  $I(P^E_{ec})$ is isomorphic to $M_{\infty}(L_K(1,l_2))$
and $I(P^F_{ec})$  is isomorphic to $M_{\infty}(L_K(1,l_2'))$. 
The unique proper nonzero graded ideals in $L_K(E)$ and $L_K(F)$ are $I(P^E_{ec})$ and $I(P^F_{ec})$, respectively (because in both
graphs, the only proper nontrivial hereditary and saturated subset is $\{v\}$).
We know, by \cite[Proof of the Theorem 5.3]{AMP}, that an ideal in a Leavitt path algebra is graded if and only if it is generated by idempotents. Therefore, graded ideals are preserved by isomorphisms. Thus, if $L_K(E)$ and $L_K(F)$ are isomorphic, by Theorem \ref{EnateMerlot-Merlot} the isomorphism maps $I(P^E_{ec})$ to $I(P^F_{ec})$. 
By Proposition \ref{uzo} we get
$l_2=l_2'$.  Moreover $L(1,l_1)\cong L_K(E)/I(P^E_{ec})\cong L_K(F)/I(P^F_{ec})\cong L(1,l_1')$, which implies $l_1=l_1'$. 

\medskip

Note also that  $t_1 = (l_2-1)q +r$ for some $q \in \N$ and $0< r \le l_2-1$.  
Observe that by applying successively shift moves to $E$, we produce $G$, where
$$  \xymatrix{E: & &
{\bullet}^u\ar@(ul,dl)_{(l_1)}  \ar@/_.5pc/[r]_{(t_1)} & {\bullet}^v
\ar@(ur,dr)^{(l_2)}    } 
\xymatrix{& G: & &
{\bullet}^u  \ar@(ul,dl)_{(l_1)} \ar@/^-.5pc/[r]_{(r)}   & {\bullet}^v
\ar@(ur,dr)^{(l_2)}  } $$
By Theorem \ref{LuisFelipe}, $L_K(E)$ is isomorphic to $L_K(G)$.
Hence, to find the isomorphism classes in Case V(b), it is enough to consider the graphs:
$$
\xymatrix{ E: & &
{\bullet}^u  \ar@(ul,dl)_{(l_1)} \ar@/^-.5pc/[r]_{(t_1)}   & {\bullet}^v
\ar@(ur,dr)^{(l_2)}  } 
 \quad 
\xymatrix{  & F: & &
{\bullet}^u  \ar@(ul,dl)_{(l_1)} \ar@/^-.5pc/[r]_{(t_1')}   & {\bullet}^v
\ar@(ur,dr)^{(l_2)}  } 
$$
where $1\le t_1,t_1'\le l_2-1$. Now, $\Delta_E = |\Delta_E| = |\Delta_F| =\Delta_F$.

If $d_E \neq d_F$, then $K_0(L_K(E))$ is not isomorphic to $K_0(L_K(F))$, and 
$L_K(E)$ cannot be isomorphic to $L_K(F)$. 

If $d_E = d_F$ and $gcd(l_1-1-t_1,l_2-1) \neq gcd(l_1-1-t_1',l_2-1) $, then $L_K(E)$ cannot be isomorphic to $L_K(F)$ as they  have different  types.
\end{proof}

\begin{remark}
\rm
We do not know if the converse of (vi) in Theorem \ref{lagardeisillaBis} is true or not.
Take $E$ and $F$ as in Case V(b). 
If $d_E = d_F$ and $gcd(l_1-1-t_1,l_2-1) = gcd(l_1-1-t_1',l_2-1)$, then 
both the type and the $K_0$ groups are the same. Moreover, when the Leavitt path algebras have type $(1,n+1)$, the order unit  will be an element of order $n$. But we do not know if this implies that the Leavitt path algebras $L_K(E)$ and $L_K(F)$ are isomorphic or not.
 Note that the graphs of this form do not produce purely infinite simple Leavitt path 
 algebras, hence we cannot use the algebraic Kirchberg-Philips Theorems. 
 \end{remark}

We illustrate that there are graphs in Case V(b) such that some of them produce Leavitt path algebras which are not isomorphic while others produce Leavitt path algebras for which we cannot say whether they are isomorphic or not.

\begin{example}\label{givine}
\rm
Consider the three graphs that follow.
$$ \xymatrix{
{\bullet}^u  \ar@(ul,dl)_4   & {\bullet}^v
\ar@(ur,dr)^2    \ar@/_.5pc/[l]_{1}}
\quad
 \xymatrix{
{\bullet}^u  \ar@(ul,dl)_4   & {\bullet}^v
\ar@(ur,dr)^2    \ar@/_.5pc/[l]_{2}}
\quad
 \xymatrix{
{\bullet}^u  \ar@(ul,dl)_4   & {\bullet}^v
\ar@(ur,dr)^2    \ar@/_.5pc/[l]_{3}}
$$

The first one produces a Leavitt path algebra of type $(1,2)$ whereas both, the second and the third graphs, produce Leavitt path algebras of type $(1,4)$ such that their $K_0$ groups are $\Z_3$. 
$$ 
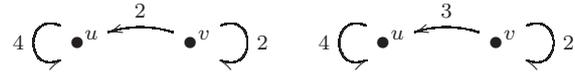
\xymatrix{
{\bullet}^u  \ar@(ul,dl)_4   & {\bullet}^v
\ar@(ur,dr)^2    \ar@/_.5pc/[l]_{2}}
\quad
 \xymatrix{
{\bullet}^u  \ar@(ul,dl)_4   & {\bullet}^v
\ar@(ur,dr)^2    \ar@/_.5pc/[l]_{3}}
$$
\end{example}
It is an open question (at least for the authors of this paper) if  the Leavitt path algebras associated to the graphs in Example \ref{givine} are isomorphic. However, we have studied if they are graded isomorphic, and the answer is no. 

\begin{example}
\rm
Consider the graphs that follow.
$$ 
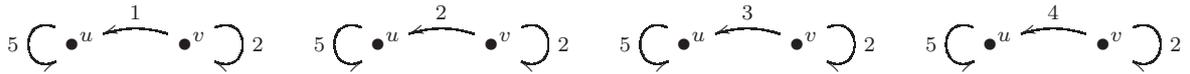
\xymatrix{
{\bullet}^u  \ar@(ul,dl)_5   & {\bullet}^v
\ar@(ur,dr)^2    \ar@/_.5pc/[l]_{1}}
\quad
 \xymatrix{
{\bullet}^u  \ar@(ul,dl)_5   & {\bullet}^v
\ar@(ur,dr)^2    \ar@/_.5pc/[l]_{2}}
\quad
 \xymatrix{
{\bullet}^u  \ar@(ul,dl)_5   & {\bullet}^v
\ar@(ur,dr)^2    \ar@/_.5pc/[l]_{3}}
\quad 
\xymatrix{
{\bullet}^u  \ar@(ul,dl)_5   & {\bullet}^v
\ar@(ur,dr)^2    \ar@/_.5pc/[l]_{4}}
$$

The first one produces a Leavitt path algebra of type $(1, 2)$,  and the third graph produces a Leavitt path algebra of type $(1,3)$. Also, the second and the fourth graphs, both produce Leavitt path algebras of type $(1,5)$, and their $K_0$ groups are the same. It is an open question whether the following graphs give rise to isomorphic Leavitt path algebras:
$$ 
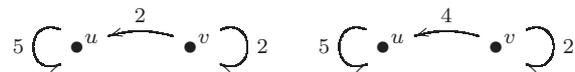
\xymatrix{
{\bullet}^u  \ar@(ul,dl)_5   & {\bullet}^v
\ar@(ur,dr)^2    \ar@/_.5pc/[l]_{2}}
\quad
 \xymatrix{
{\bullet}^u  \ar@(ul,dl)_5   & {\bullet}^v
\ar@(ur,dr)^2    \ar@/_.5pc/[l]_{4}}
$$

\end{example}


\section{Classification of Leavitt path algebras having IBN} 


In the previous section we have classified the Leavitt path algebras not having IBN. Now we complete the classification of Leavitt path algebras associated to finite graphs having two vertices  by describing Leavitt path algebras having IBN. We list them in the same order of dichotomies outlined in Figure \ref {WeAreHappy}.  
The  invariant ideals of the families of Leavitt path algebras having IBN are summarized in Figure \ref{muge2}. 
\begin{figure}[H]
\resizebox{17cm}{!}{
\begin{tabular}{|c|c|c|cc|c|c|c|c|}
\hline
 &DEC & PIS & Graph & & $I(P_l)$ &  $I(P_c)$ & $I(P_{ec})$ & $L_K(E)/I$ \\
\hline
&&&&&&&&\\
A1 & YES & NO & $ \xymatrix{{\bullet}^u    & {\bullet}^v     }
$ & & $K \times K$ & $0$ & $0$ &$0$\\
&&&&&&&&\\

 A2 & YES & NO & $ \xymatrix{{\bullet}^u    & {\bullet}^v\ar@(ur,dr)^{(1)}     }
$& & $K$ & $K[x,x^{-1}]$ & $0$ &$0$\\
&&&&&&&&\\
 A3 & YES & NO & $$ \xymatrix{
{\bullet}^u    & {\bullet}^v\ar@(ur,dr)^{(l_2)}     }
$$& $ l_2\ge 2$ & $K$ & $0$ & $ L(1,l_2)$ &$0$\\
&&&&&&&&\\
A4 & NO  & NO &$$\xymatrix{
{\bullet}^u   & {\bullet}^v
    \ar@/_.5pc/[l]_{(t_2)}}
$$ &  $t_2\ge 1$ & $ \M_{t_2+1}(K)$ & $0$ & $0$ & $0$ \\
&&&&&&&&\\
 A5 & NO  & NO &$$\xymatrix{
{\bullet}^u   & {\bullet}^v
\ar@(ur,dr)^{(1)}    \ar@/_.5pc/[l]_{(t_2)}}
$$ & $t_2\ge 1$ & $ \M_\infty(K)$ & $0$ & $0$ &$K[x,x^{-1}]$  \\
&&&&&&&&\\

 A6 & NO  & NO &$$\xymatrix{
{\bullet}^u   & {\bullet}^v
\ar@(ur,dr)^{(l_2)}    \ar@/_.5pc/[l]_{(t_2)}}
$$ & $l_2-t_2 \ne 1$; $l_2 > 1$; $t_2\ge 1$ & $ \M_\infty(K)$ & $0$ & $0$ &$L(1,l_2)$  \\
&&&&&&&&\\

 A7 & YES & NO & $$ \xymatrix{
{\bullet}^u\ar@(ul,dl)_{(1)}    & {\bullet}^v\ar@(ur,dr)^{(1)}     }
$$ & & 0 & $K[x,x^{-1}] \times K[x,x^{-1}]$ & $0$ &$0$\\
&&&&&&&&\\

A8 &YES &  NO & $$ \xymatrix{
{\bullet}^u\ar@(ul,dl)_{(1)}    & {\bullet}^v\ar@(ur,dr)^{(l_2)}     }
$$ & $l_2 \ge 2$& 0 & $K[x,x^{-1}]$  & $L(1,l_2)$ & $0$\\
&&&&&&&&\\

A9 &NO &  NO & $$ \xymatrix{
{\bullet}^u\ar@(ul,dl)_{(1)}    & {\bullet}^v \ar@/_.5pc/[l]_{(t_2)}}
$$ & $t_2 \geq 1$& 0 & $M_{t_2+1}(K[x,x^{-1}])$  & 0 & $0$\\
&&&&&&&&\\

A10 &NO &  NO & $$ \xymatrix{
{\bullet}^u\ar@(ul,dl)_{(1)}    & {\bullet}^v \ar@/_.5pc/[l]_{(t_2)} \ar@(ur,dr)^{(1)}     }
$$ & $t_2 \geq 1$& 0 & $M_\infty(K[x,x^{-1}])$  & $0$ & $K[x,x^{-1}]$\\
&&&&&&&&\\

A11 &  NO & NO    & $$\xymatrix{
{\bullet}^u\ar@(ul,dl)_{(1)}   & {\bullet}^v
\ar@(ur,dr)^{(l_2)}    \ar@/_.5pc/[l]_{(t_2)}}$$ &  $l_2-t_2\ne 1;  l_2 \geq 2, t_2 \geq 1$ & $0$ &  $\M_\infty(K[x,x^{-1}])$ & $0$& $L(1,l_2)$
\\ 
&&&&&&&&\\
 
A12 & NO & NO   & 
$$ \xymatrix{
{\bullet}^u   \ar@/^-.5pc/[r]_{(1)}   & {\bullet}^v \ar@/_.5pc/[l]_{(1)}
   }
$$ & 
 & $0$ & $\M_2(K[x,x^{-1}])$ & $0$ &$0$\\
&&&&&&&&\\

A13 & NO & NO   & 
$$ \xymatrix{
{\bullet}^u   \ar@(ul,dl)_{(1)} \ar@/^-.5pc/[r]_{(t_1)}   & {\bullet}^v \ar@(ur,dr)^{(l_2)}
}
$$ &$ l_2 \geq 2; t_1\geq 1$
 & $0$ & $0$ & $ \M_\infty(L(1,l_2))$ & $ K[x,x^{-1}]$\\
&&&&&&&&\\

A14 & NO & YES    & 
$$ \xymatrix{
{\bullet}^u  \ar@(ul,dl)_{(l_1)} \ar@/^-.5pc/[r]_{(t_1)}   & {\bullet}^v
\ar@(ur,dr)^{(l_2)}    \ar@/_.5pc/[l]_{(t_2)}}
$$ &
$\Delta_E = 0$ & $0$ & $0$ & $L_K(E)$ &$0$\\
&&&&$l_i-t_i\ne 1; t_i, l_i\geq 1, \ i=1,2$ &&&&\\
\hline
\end{tabular}}

\medskip

\caption{Invariants for two-vertex Leavitt path algebras having IBN}\label{muge2}
\end{figure}
In a Leavitt path algebra having IBN, clearly the order of the order unit  $[1]_E$ will be infinite and hence $K_0$ contains an infinite subgroup. 
We now compute the $K_0$ groups of each family. Note that $K_0$ is isomorphic to $\Z \times \Z_{d_E}$, 
where $d_E$ is the greatest common divisor of the entries of the matrix $N'_E$ and $\Delta_E$ is $0$ in all the families A1-A14.  

\begin{figure}[H]
\resizebox{8cm}{!}{
\begin{tabular}{|c|c|c|c|c|}
\hline
 &  $A_E$ &  $1_E'$ & $N'_E$ & $K_0$ \\
\hline
&&&&\\
A1 & $\begin{pmatrix} 0 & 0 \\ 0 & 0  \end{pmatrix}$ & $\begin{pmatrix} 0 & 0 \\ 0 & 0  \end{pmatrix}$ & $\begin{pmatrix} 0 & 0 \\ 0 & 0  \end{pmatrix}$ &$\Z \times \Z$\\
&&&&\\
 A2  & $\begin{pmatrix} 0 & 0 \\ 0 & 1  \end{pmatrix}$ & $\begin{pmatrix} 0 & 0 \\ 0 & 1  \end{pmatrix}$ & $\begin{pmatrix} 0 & 0 \\ 0 & 0  \end{pmatrix}$ &$\Z \times \Z$\\
&&&&\\
 A3  & $\begin{pmatrix} 0 & 0 \\ 0 & l_2  \end{pmatrix}$ & $\begin{pmatrix} 0 & 0 \\ 0 & 1  \end{pmatrix}$ 
& $\begin{pmatrix} 0 & 0 \\ 0 & l_2-1  \end{pmatrix}$ &$\Z \times \Z_{l_2-1}$\\
&&&&\\
A4  & $\begin{pmatrix} 0 & 0 \\ t_2 & 0  \end{pmatrix}$ & $\begin{pmatrix} 0 & 0 \\ 0 & 1  \end{pmatrix}$ & 
$\begin{pmatrix} 0 & 0 \\ t_2 & -1  \end{pmatrix}$ &  $\Z$\\
&&&&\\
 A5 & $\begin{pmatrix} 0 & 0 \\ t_2 & 1  \end{pmatrix}$ & $\begin{pmatrix} 0 & 0 \\ 0 & 1  \end{pmatrix}$ & 
$\begin{pmatrix} 0 & 0 \\ t_2 & 0  \end{pmatrix}$ &  $\Z \times \Z_{t_2}$ \\
&&&&\\
 A6  &  $\begin{pmatrix} 0 & 0 \\ t_2 & l_2  \end{pmatrix}$ & $\begin{pmatrix} 0 & 0 \\ 0 & 1  \end{pmatrix}$ 
& $\begin{pmatrix} 0 & 0 \\ t_2 & l_2-1  \end{pmatrix}$ &  $\Z \times \Z_{gcd(t_2,l_2-1)}$\\
&&&&\\
 A7 & $\begin{pmatrix} 1 & 0 \\ 0 & 1  \end{pmatrix}$ & $\begin{pmatrix} 1 & 0 \\ 0 & 1  \end{pmatrix}$ & $\begin{pmatrix} 0 & 0 \\ 0 & 0  \end{pmatrix}$ &$\Z \times \Z$\\
\hline
\end{tabular}}
\quad
\resizebox{8cm}{!}{
\begin{tabular}{|c|c|c|c|c|}
\hline
 &  $A_E$ &  $1_E'$ & $N'_E$ & $K_0$ \\
\hline
&&&&\\
A8 & $\begin{pmatrix} 1 & 0 \\ 0 & l_2  \end{pmatrix}$ & $\begin{pmatrix} 1 & 0 \\ 0 & 1  \end{pmatrix}$  
& $\begin{pmatrix} 0 & 0 \\ 0 & l_2-1  \end{pmatrix}$ & $\Z \times \Z_{l_2-1}$\\
&&&&\\
A9  & $\begin{pmatrix} 1 & 0 \\ t_2 & 0  \end{pmatrix}$ & $\begin{pmatrix} 1 & 0 \\ 0 & 1  \end{pmatrix}$  & 
$\begin{pmatrix} 0 & 0 \\ t_2 & -1  \end{pmatrix}$ & $\Z$ \\
&&&&\\
A10  & $\begin{pmatrix} 1 & 0 \\ t_2 & 1  \end{pmatrix}$ & $\begin{pmatrix} 1 & 0 \\ 0 & 1  \end{pmatrix}$  & $\begin{pmatrix} 0 & 0 \\ t_2 & 0  \end{pmatrix}$ & $\Z \times \Z_{t_2}$\\
&&&&\\
A11 & $\begin{pmatrix} 1 & 0 \\ t_2 & l_2  \end{pmatrix}$ &  $\begin{pmatrix} 1 & 0 \\ 0 & 1  \end{pmatrix}$ & $\begin{pmatrix} 0 & 0 \\ t_2 & l_2-1 \end{pmatrix}$& $\Z \times \Z_{gcd(t_2,l_2-1)}$\\ 
&&&&\\
A12 &   $\begin{pmatrix} 0 & 1 \\ 1 & 0  \end{pmatrix}$ & $\begin{pmatrix} 1 & 0 \\ 0 & 1  \end{pmatrix}$ & 
$\begin{pmatrix} -1 & 1 \\ 1 & -1  \end{pmatrix}$ & $\Z$\\
&&&&\\
A13 
 & $\begin{pmatrix} 1 & t_1 \\ 0 & l_2  \end{pmatrix}$& $\begin{pmatrix} 1 & 0 \\ 0 & 1  \end{pmatrix}$ &  $\begin{pmatrix} 0 & t_1 \\ 0 & l_2-1  \end{pmatrix}$ & $\Z \times \Z_{gcd(t_1,l_2-1)}$\\
&&&&\\
A14  & $\begin{pmatrix} l_1 & t_1 \\ t_2 & l_2  \end{pmatrix}$ & $\begin{pmatrix} 1 & 0 \\ 0 & 1  \end{pmatrix}$ & $\begin{pmatrix} l_1-1 & t_1 \\ t_2 & l_2-1  \end{pmatrix}$ & $\Z \times \Z_{d_E}$ \\
&&&& \\
\hline
\end{tabular}}
\medskip
\caption{$K_0$ for two-vertex Leavitt path algebras having IBN}\label{muge3}
\end{figure}


By looking at the invariants in Figure \ref{muge2}, it is easily deducible that two Leavitt path algebras whose graphs are in different families from A1 through A14 are non-isomophic. Now, we study the isomorphisms within the Leavitt path algebras in each family. 
Every graph in A1, A2, A3, A4, A7, A8, A9 and A12 produces a unique Leavitt path algebra which is not isomorphic to any other. 
By looking at Figure \ref{muge3}, the $K_0$ group in the classes A5 and A10 is $\Z \times \Z_{t_2}$ and each graph produce a distinct isomorphism class again.   
\bigskip

In A6 and A11, the Leavitt path algebra $L_K(E)/I(P_{lec})$ is isomorphic to $L(1,l_2)$ this assures that any two graphs from the same family giving rise to isomorphic Leavitt path algebras must have the same $l_2$. By looking at Figure \ref{muge3}, for distinct $t_2, t_2'$, with $l_2-t_2, l_2-t_2' \ne 1$, if 
$gcd(t_2,l_2-1) \ne gcd(t_2',l_2-1)$, the $K_0$ groups are non-isomorphic, hence they produce different isomorphism classes. 
However, if  $gcd(t_2,l_2-1) = gcd(t_2',l_2-1)$, then we do not know whether the corresponding graphs produce isomorphic Leavitt path algebras. 

\bigskip

In the group A13, similarly, the invariant ideal $I(P_{ec})$ of $L_K(E)$ is isomorphic to $\M_\infty(L(1,l_2))$ and hence  any two graphs in this group having isomorphic Leavitt path algebras will have the same $l_2$, by Proposition \ref{uzo}. 
For distinct $t_1, t_1'$, if 
$gcd(t_1,l_2-1) \ne gcd(t_1',l_2-1)$, the $K_0$ groups are non-isomorphic, hence producing different isomorphism classes. 
However, if  $gcd(t_1,l_2-1) = gcd(t_1',l_2-1)$, then we do not know whether the corresponding graphs produce isomorphic Leavitt path algebras. 

\bigskip

The family A14 contains Leavitt path algebras having IBN that are purely infinite simple.

For any Leavitt path algebra 
$L_K(E)$, where $E$ is in the family A14, $\Delta_E=0$, so the Franks triple determines when the graphs belonging to A14 induce isomorphic Leavitt path algebras. 

Now, we can state the following theorem, that we have proved above.

\begin{theorem}\label{tintilladerotaBis}
Let $E$ be a finite graph with two-vertices whose Leavitt path algebra $L_K(E)$ has IBN. Then, $L_K(E)$ is isomorphic to a Leavitt path algebra whose associated graph is one in Cases A1-A14. Moreover:
\begin{enumerate}[\rm (i)]
\item In each of the Cases A1, A2, A7 and A12 the Leavitt path algebra  is isomorphic to $K\times K$, $K\times K[x,x^{-1}]$, $K[x,x^{-1}]\times K[x,x^{-1}]$ and ${\mathcal M}_2(K[x,x^{-1}])$, respectively.

\item In Case A3 the Leavitt path algebra is isomorphic to $K\times L(1,l_2)$. Two Leavitt path algebras $K\times L(1,l_2)$ and $K\times L(1,l'_2)$ are isomorphic if and only if $l_2=l'_2$.

\item  In Case A4 the Leavitt path algebra is isomorphic to ${\mathcal M}_{t_2+1}(K)$. Two Leavitt path algebras ${\mathcal M}_{t_2+1}(K)$ and ${\mathcal M}_{t'_2+1}(K)$ are isomorphic if and only if $t_2=t'_2$.

\item In Cases A5 and A10, the Leavitt path algebras are determined by their $K_0$ groups.

\item For every graph in Case A8 the associated Leavitt path algebra is decomposable and isomorphic to  $K[x, x^{-1}]\times L(1,l_2)$, for $l_2 \geq 2$. The isomorphisms are determined by the value of $l_2$.

\item In Case A9 the associated Leavitt path algebra is isomorphic to $M_{t_2+1}(K[x, x^{-1}])$. The isomorphisms are determined by $t_2$.

\item In Case A14 the associated Leavitt path algebra is purely infinite simple (in fact, it is the only one purely infinite simple  having IBN). Two Leavitt path algebras associated to graphs in this case are isomorphic if and only if their Franks triples coincide because the determinant is zero.

\item In Cases A6, A11 and A13 two different graphs $E$ and $F$ with $d_E \ne d_F$ give rise to non-isomorphic Leavitt path algebras.

\end{enumerate}
\end{theorem}

We conclude this paper with the following open question which will complete the full classification if answered either affirmative or negative.  

\begin{question}
\rm
Given any two graphs $E$ and $F$ with $d_E = d_F$,  
either in the same family $\mathcal{C}$, where $\mathcal{C} \in \{ A6, A11, A13 \}$, or 
in the family of V(b), with associated sets $S_E=\{(l_1,t_1),(l_2,0)\}$ and $S_F=\{(l_1,t_1'),(l_2,0)\}$ 
having $gcd(l_1-1-t_1, l_2-1) = gcd(l_1-1-t_1', l_2-1)$, are the Leavitt path algebras $L_K(E)$ and $L_K(F)$ isomorphic?  
\end{question}

\section*{Acknowledgments}

The authors would like to thank Adam Peder Wie S{\o}rensen for  useful comments.


\begin{thebibliography}{10}

\bibitem{AALP} \textsc{Gene Abrams, Pham Ngoc \'Anh, Adel Louly, Enrique Pardo}, The classification question for Leavitt path algebras, \emph{J. Algebra} \textbf{320} (2008), 1983--2026.

\bibitem{Flow} \textsc{Gene Abrams, Adel Louly, Enrique Pardo, Cristopher Smith}. 
Flow invariants in the classification of Leavitt path algebras.
\emph{Journal of Algebra}. \textbf{333 (1)} (2011), 202--231.

\bibitem{AAS} \textsc{Gene Abrams, Pere Ara, Mercedes Siles Molina}, \textit{Leavitt path algebras}, Lecture Notes in Mathematics, Springer (2017). 

\bibitem{AA1} \textsc{Gene Abrams, Gonzalo Aranda Pino}, The Leavitt path algebra of a graph, \emph{J. Algebra}, \textbf{293} (2005), 319--334. 

\bibitem{AAS2} \textsc{Gene Abrams, Gonzalo Aranda Pino, Mercedes Siles Molina}, Locally finite Leavitt path algebras, \emph{Israel J. Math.}, \textbf{165} (2008), 329--348. 


\bibitem{atlas} \textsc{Pablo Alberca Bjerregaard, Gonzalo Aranda Pino,  Dolores Mart\'{\i}n Barquero, C\'andido Mart\'{\i}n Gonz\'{a}lez, Mercedes Siles Molina}, Atlas of Leavitt path algebras of small graphs. 
\emph{J. Math. Soc. Japan}  \textbf{66 (2)} (2014), 581--611.

\bibitem{AB} \textsc{Pere Ara, Miquel Brustenga}, Module theory over Leavitt path algebras and $K$-theory, \emph{J. Pure Appl. Algebra} \textbf{214 (7)} (2010), 1131--1151.

\bibitem{AMP} \textsc{Pere Ara, Mar\'{\i}a  \'Angeles Moreno,  Enrique Pardo}, Nonstable K-theory for Graph Algebras. \emph{Algebr. Representation Theor.} \textbf{10} (2007), 157--178.

\bibitem{ABS} \textsc{Gonzalo Aranda Pino,  Jose Brox, 
Mercedes Siles Molina}, Cycles in Leavitt path algebras by means of idempotents.
\emph{Forum Mathematicum}, \textbf{27}  (2015), 601--633.  

\bibitem{AC} \textsc{Gonzalo Aranda Pino,  Kathi  Crow}, 
The center of a Leavitt path algebra. \emph{Rev. Mat. Iberoam.} \textbf{27}(2) (2011), 621--644.




\bibitem{CMMS}\textsc{Lisa O. Clark, Dolores Mart\'{\i}n Barquero, C\'andido Mart\'{\i}n Gonz\'alez, Mercedes Siles Molina.}  Using Steinberg algebras to study decomposability of Leavitt path algebras. \emph{Forum Mathematicum}. Published Online: 2016-12-14 | DOI: https://doi.org/10.1515/forum-2016-0062.

\bibitem{Maria} \textsc{Mar\'\i a G. Corrales Garc\'\i a}, {Nuevas aportaciones al estudio de las \'algebras de caminos de Leavitt. Doctoral dissertation. Universidad de M\'alaga} (2016).

\bibitem{CMMSS2} \textsc{Mar\'ia G. Corrales Garc\'\i a,  Dolores Mart\'{\i}n Barquero, C\'andido Mart\'{\i}n Gonz\'{a}lez, 
Mercedes Siles Molina, Jos\'e F. Solanilla Hern\'andez}, Extreme cycles. The center of a Leavitt path algebra. \emph{Pub. Mat.} \textbf{60} (2016), 235--263.




\bibitem{HH} \textsc{Brian Hartley, Trevor O. Hawkes},
\emph{Rings, Modules and Linear Algebra}, Chapman $\&$ Hall, London-New York,  xi+210 pp. (1980).


\bibitem{KO} \textsc{M\"uge Kanuni, Murad \"Ozaydin}, Cohn-Leavitt path algebras and the invariant basis number property. \emph{To appear.} arXiv:1606.07998v1.

\bibitem{L} \textsc{William G. Leavitt}, The module type of a ring.
\emph{Trans. Amer. Math. Soc.} \textbf{103} (1962), 113--130.


\bibitem{Smith} \textsc{Henry J. Stephen Smith}, On systems of linear indeterminate equations and congruences. \emph{Philosophical Trans. of the Royal Soc. of London} \textbf{151} (1861), 293--326.

\bibitem{SZE} \textsc{Hu Sze-Tsen},
\emph{Introduction to homological algebra}, Holden-Day INC., San Francisco (1968).

\end{thebibliography}
\end{document}